\documentclass[12pt]{amsart}
\usepackage{amsmath,amssymb,txfonts}

\usepackage{times}
\usepackage{bbm}
\usepackage[dvips]{graphicx}

\allowdisplaybreaks
\usepackage{amsxtra}
\usepackage[cp1252]{inputenc}
\usepackage{color}

\usepackage{mathrsfs}

\numberwithin{equation}{section}

\newcommand{\eps}[0]{\varepsilon}
\newcommand{\ave}[1]{\langle #1\rangle}
\newcommand{\wt}[1]{\widetilde{#1}}
\newcommand{\QQ}{\mathbb{Q}}
\newcommand{\supp}{{\rm supp}}
\newcommand{\R}{\mathbb R}

\newtheorem{theorem}{Theorem}[section]
\newtheorem*{TheoremLetter}{Theorem A}

\newtheorem{remark}[theorem]{Remark}
\newtheorem{corollary}[theorem]{Corollary}
\newtheorem{proposition}[theorem]{Proposition}
\newtheorem{defi}{\hskip\parindent {Definition}}[section]
\newtheorem{lemma}[theorem]{Lemma}

\newcommand{\ct}[1]{\langle {#1}\rangle \lower.3ex\mathrm{$_{t}$}}
\newcommand{\lt}[1]{[ {#1}] \lower.3ex\mathrm{$_{t}$}}

\DeclareMathOperator{\BMO}{BMO}

\newcommand{\CC}{\mathbb{C}}

\numberwithin{equation}{section}

\begin{document}

	\title[Fractional Bloom boundedness of commutators]{Fractional Bloom boundedness of commutators in spaces of homogeneous type}
	\author{Zhenbing Gong}
	\address{Zhenbing Gong, Department of Applied Mathematics, School of Mathematics and Physics,
		Southwest University of Science and Technology,
		Sichuan 621010, China
	}
	\email{zhenbinggong@swust.edu.cn}
	
	\author{Ji Li}
	\address{Ji Li, Department of Mathematics, Macquarie University, NSW, 2109, Australia}
	\email{ji.li@mq.edu.au}
	
	\author{Jaakko Sinko}
	\address{Jaakko Sinko, Department of Mathematics and Statistics, P.O.B. 68 (Pietari Kalmin katu 5), FI-00014 University of Helsinki, Finland}
	\email{jaakko.sinko@helsinki.fi}

	\date{}

	\bigbreak

	\begin{center}
	\end{center}

	\renewcommand{\thefootnote}{}
	
%
%

	\begin{abstract}
		We aim to characterise boundedness of commutators $[b,T]$ of singular integrals $T$. Boundedness is studied between weighted Lebesgue spaces $L^p(X)$ and $L^q(X)$, $p\leq q$, when the underlying space $X$ is a space of homogeneous type. Commutator theory in spaces of homogeneous type already exist in literature, in particular boundedness results in the setting $p=q$. The purpose here is to extend the earlier results to the setting of $p< q$.
  
        Our methods extend those of Duong et al.\ \cite{DGKLWY2021} and Hyt\"{o}nen et al.\ \cite{HOS2023}. A novelty here is that in order to show the lower bound of the commutator norm, we demonstrate that the approximate weak factorisation of Hyt\"{o}nen can be used when the underlying setting is a space of homogeneous type and not only in the Euclidean setting. The strength of the approximate weak factorisation is that (when compared to the so-called median method) it readily allows complex-valued $b$ in addition to real-valued ones. However, the median method has been previously successfully applied to iterated commutators and thus has its own strengths. We also present a proof based on that method.
		
	\end{abstract}
	
	\maketitle

	\arraycolsep=1pt

    \tableofcontents
    
\section{Introduction}

In the early 1970s, Coifman and Weiss introduced the spaces of homogeneous type \cite{CW1971,CW1977}, which generalize the Euclidean space as a setting for harmonic analysis.
We recall that $(X, d, \mu)$ is a space of homogeneous type if $d\colon X\times X\to [0,\infty)$ is a quasi-metric on $X$, namely,
\begin{itemize}
    \item[(i)] for all $x, y \in X, d(x, y) = d(y, x)$,
    \item[(ii)] for all $x, y \in X, d(x, y) \geq 0$ and $d(x, y) = 0$ if and only if $x = y$,
    \item[(iii)] there exists a constant $A_0 \ge 1$ such that $d(x, y) \le A_0(d(x,z) + d(z, y))$ for all $ x, y,z \in X$,
\end{itemize}
and if $\mu$ is a nonnegative measure defined on the Borel sets of $X$ that satisfies the doubling condition, that is, there exists a
constant $A_1 \ge 1$ such that for all $x\in X,r > 0$,
\begin{align}\label{Bxy}
0<\mu(B(x, 2 r)) \leq A_1 \mu(B(x, r))<\infty,
\end{align}
where $B(x, r):=\{y \in X: d(x, y)<r\}$ is a ball with center $x$ and radius $r$. A subset $U\subset X$ is declared open if for every $x\in U$ there exists $r>0$ such that $B(x,r)\subset U$; this defines the topology (and Borel sets) in $X$.
If $C_{\mu}$ is the smallest constant such that \eqref{Bxy} holds, then we call $\mathbb {Q}:= \log_{2} C_{\mu}$ (the upper dimension of $\mu$) the doubling
order of $\mu$. In fact, \eqref{Bxy} implies that for all $x \in X, \lambda \ge 1$ and $r > 0$,
$$\mu(B(x, \lambda r)) \le A_1 \lambda^\mathbb{Q}\mu(B(x,r)).$$
Throughout this paper we assume that $\mu(X) = \infty$ and that $\mu(\{x_0\}) = 0$ for every $x_0\in X$.

We spare a moment here for a note on the $\mu$-measurable sets. Naturally, for \eqref{Bxy} to make sense, at least all balls $B(x,r)$ have to be measurable. Hence, in our definition of a space of homogeneous type we especially require that all balls are Borel sets. It is very likely that our main results remain valid even if this regularity assumption is somewhat weakened. We postpone elaborating on this topic to Section \ref{discussion}.

Research on the boundedness of commutators of singular integral operators on spaces of homogeneous type will be traced back to Krantz and Li's 2001 paper \cite{KL2001}.  
However, on Euclidean spaces, a great deal of research has been done.
Let $X$ be a space of homogeneous type. For a function $f$ on $X$, the commutator $[b, T]$ of the singular integral operator $T$ with a symbol
$b$, which is defined by
\begin{align*}
[b, T]f(x) = b(x)Tf(x) - T(bf)(x),
\end{align*}
has played a vital role in harmonic analysis, complex analysis, and partial
differential equations. 
It is well-known that Coifman, Rochberg and Weiss \cite{CRW1976} provided a seminal characterization of the boundedness of the commutator $\left[b, R_i\right]$ acting on Lebesgue spaces, where $R_j=\frac{\partial}{\partial x_j} \Delta^{-1 / 2}$ denotes the $j$-th Riesz transform on the Euclidean space $\mathbb{R}^n$. This characterization was established in terms of BMO, extending Nehari's work \cite{Ne1957} on Hankel operators from the complex setting to the real setting of $\mathbb{R}^n$.
See \cite{CHDD2020,CYZ2022,J1978,LLO2024} for some recent and not-so-recent research
development of the commutator $[b,T]$ on the Euclidean space $\mathbb{R}^n$ and \cite{DY2003,PS2007} on the spaces of homogeneous
type.

 In 1985, Bloom \cite{B1985} proves a two-weight extension of the Nehari \cite{Ne1957} result in one dimension. This extension entailed the characterization of weighted BMO  in terms of boundedness of commutators $[b, H]$ in the two-weight setting, where $H$ represents the Hilbert transform on $\mathbb{R}$. Bloom's result has been extended by many scholars. 
In the study commonly referred to as ``of Bloom type'', Holmes--Lacey--Wick \cite{HLW2016,HLW2017} have made significant advancements, providing a characterization of weighted BMO space on $\mathbb{R}^n$ in terms of the boundedness of commutators of Riesz transforms in the two weight setting, by using the methods established by Petermichl \cite{P2000} for the
Hilbert transform, and by Hyt\"{o}nen \cite{H2012} for general Calder\'{o}n--Zygmund operators.
Later, Lerner--Ombrosi--Rivera-R\'{\i}os \cite{LOR2017,LOR2019} characterized the boundedness of commutators of Calder\'{o}n--Zygmund operators with homogeneous kernels $\Omega\left(\frac{x}{|x|}\right) \frac{1}{|x|^n}$  in the two-weight setting. To complete the picture, recently K. Li \cite{L2022} extended the recently well-studied two-weight commutator estimates to the multilinear setting. Further, the second author and his collaborators \cite{DGKLWY2021} considered commutators in the two-weight setting on spaces of homogeneous type. We refer the reader to \cite{DGKLWY2021} for several examples of singular integral operators and their commutators whose natural settings are examples of spaces of homogeneous type beyond the Euclidean setting of $\R^n$ equipped with the Euclidean distance and Lebesgue measure.

Recently, the third author and his collaborators \cite{HOS2023} established the fractional ($p\leq q$) Bloom boundedness of commutators on $\mathbb{R}^n$. Findings are as follows (for notation, see \cite{HOS2023} and Section \ref{section 2} below):
\begin{TheoremLetter}[\cite{HOS2023}]
Let $T$ be a non-degenerate Calder\'{o}n-Zygmund operator with $\omega$ satisfying the Dini condition and let $b \in L_{\mathrm{loc}}^1(\mathbb{R}^n)$. Let $1<p \leq q<\infty$, let $\alpha / n=1 / p-1 / q$, let $\lambda_1 \in A_{p, p}$ and $\lambda_2 \in A_{q, q}$ and let $\nu=\nu_{p, q}$ be the fractional Bloom weight as in Definition \ref{defi1.3}. Then, it holds that
$$
\|[b, T]\|_{L_{\lambda_1}^p(\R^n) \rightarrow L_{\lambda_2}^q(\R^n)} \approx \|b\|_{\mathrm{BMO}_\nu^\alpha(\R^n)}
$$
and
$$
[b, T] \colon L_{\lambda_1}^p(\R^n) \to  L_{\lambda_2}^p(\R^n) \text{ is compact} \quad \text { iff } \quad b \in \mathrm{VMO}_\nu^\alpha(\R^n).
$$
\end{TheoremLetter}
Then, it is natural to study the following question: Can one establish the characterisation
of boundedness of commutators in the fractional Bloom setting in terms of the related weighted
BMO space for Calder\'{o}n-Zygmund operators $T$ in spaces of homogeneous type? In this paper, we provide a affirmative response to the unresolved issue and the primary outcome can be articulated as follows:

\begin{theorem}\label{thm1}
 Let $1<p \leq q<\infty$, $\frac{\alpha}{\mathbb Q}=\frac{1}{p}-\frac{1}{q}$, $\lambda_1 \in A_{p, p}$ and $\lambda_2 \in A_{q, q}$ and let $\nu=\nu_{p, q}$ be the fractional Bloom weight as in Definition \ref{defi1.3}. Suppose $b\in \BMO_\nu^\alpha(X)$. Then for any Calder\'{o}n-Zygmund operator $T$ as in Definition \ref{defi1.2}
with $\omega$ satisfying the Dini condition, there exists a positive constant $C$ such that
$$
\|[b, T]\|_{L_{\lambda_1}^p(X) \rightarrow L_{\lambda_2}^q(X)} \leq C \|b\|_{\BMO_\nu^\alpha(X)}.
$$
\end{theorem}

\begin{theorem}\label{thm2}  Suppose  $b\in L^1_{loc}(X)$. Suppose $T$ is a Calder\'{o}n-Zygmund operator as in Definition \ref{defi1.2} with $\omega$ satisfying the Dini condition and suppose $T$ satisfies the ``non-degenerate" condition \eqref{eq-1.1}.
 Let $1<p \leq q<\infty$, $\frac{\alpha}{\mathbb Q}=\frac{1}{p}-\frac{1}{q}$, $\lambda_1 \in A_{p, p}$ and $\lambda_2 \in A_{q, q}$ and let $\nu=\nu_{p, q}$ be the fractional Bloom weight as in Definition \ref{defi1.3}. If $[b,T]$ is a bounded operator from $L_{\lambda_1}^p(X)$ to $L_{\lambda_2}^q(X)$, then $b\in \BMO_\nu^\alpha(X)$ and there exists a positive constant $C$ such that
$$
\|b\|_{\BMO_\nu^\alpha(X)}\leq C\|[b, T]\|_{L_{\lambda_1}^p(X) \rightarrow L_{\lambda_2}^q(X)}.
$$
\end{theorem}

In a follow-up paper, we will characterize the compactness of $[b,T]\colon L_{\lambda_1}^p(X) \rightarrow L_{\lambda_2}^q(X)$. This will complete the generalisation of the main results of \cite{HOS2023} to the setting of spaces of homogeneous type.

In the proof of Theorem A in \cite{HOS2023}, the constant $C$ for the bound 
\[
\|[b, T]\|_{L_{\lambda_1}^p(\R^n) \rightarrow L_{\lambda_2}^q(\R^n)} \leq C\|b\|_{\mathrm{BMO}_\nu^\alpha(\R^n)}
\]
may be tracked. We refer the reader to \cite{HLS2023} for a constant $C$ whose dependence on the weight characteristics $[\lambda_1]_{A_{p,p}}$ and $[\lambda_2]_{A_{q,q}}$ is an improvement over that of \cite{HOS2023}.

This paper is organised as follows. In Section \ref{section 2}, we give some notations and recall the necessary preliminaries
on spaces of homogeneous type. In Section \ref{section 3}, we obtain the upper bound of the commutator on spaces of homogeneous type, i.e., Theorem \ref{thm1}, by using the sparse pointwise domination. In Section \ref{section 4}, we provide the lower bound of the commutator, i.e., Theorem \ref{thm2}, by using the median method or approximate weak factorisation. 

Throughout the paper, we denote by $C$ positive constants that may vary from line to line and depend at most on the parameters that are considered fixed. If $f \le Cg$ or $f\ge Cg$,we then write $f\lesssim g$ or $f\gtrsim g$; and if $f\lesssim g \lesssim f$, we write $f\approx g$. If we want to emphasise the parameters that the constant $C$ depends at most on, we write $f\lesssim_{A_0} g$ or $f\lesssim C_{A_0} g$, for example.

\section{Notation and preliminaries on spaces of homogeneous type}\label{section 2}

\subsection{Notation}The most used notation is summarized in the following table:

\begin{tabular}{c p{0.7\textwidth}}
$X$ & A space of homogeneous type. \\
$d$ & The quasi-metric of the space of homogeneous type $X$. \\
$\mu$ & The doubling measure of the space of homogeneous type $X$. \\
$p'$ & Conjugate exponent of $p\in(1,\infty)$: $\frac{1}{p}+\frac{1}{p'}:=1$.\\
$\alpha$ & Exponent defined by $\frac{\alpha}{\QQ}:=\frac{1}{p}-\frac{1}{q}$.\\
$\lambda_1$ & An $A_{p,p}$ weight in $X$.  \\
$\lambda_2$ & An $A_{q,q}$ weight in $X$.  \\
$\nu$ & Bloom weight defined by
$
\nu := (\lambda_1 / \lambda_2)^\frac{1}{1/p+1/q'}.
$ \\

$\chi_E$ & Indicator function of the set $E\subset X$. \\
$\ave{f}_E$ & Average: $\ave{f}_E:=\frac{1}{\mu(E)}\int_E f(x) d \mu(x).$ \\
$\langle f,g \rangle$ & Integral pairing: $\langle f,g \rangle:=\int_X f(x)g(x) d \mu(x)$. \\
$w(E)$ & $w(E):=\langle w, \chi_E\rangle$. \\
$\|b\|_{\BMO_{\nu}^{\alpha}(X)}$ & $\|b\|_{\BMO_{\nu}^{\alpha}(X)}:=\sup_{B}\frac{1}{\nu(B)^{\frac{\alpha}{\mathbb Q}}}\bigg(\frac{1}{\nu(B)}\int_{B}|b(x)-\ave{b}_{B}|d\mu(x) \bigg)$ \\
$\|f\|_{L^p_w(X)}$ & $\|f\|_{L^p_w(X)}:=\left(\int_{X} |f(x)w(x)|^p d \mu(x) \right)^{1/p}$. \\
$K^*$ & $K^*(x,y):=K(y,x)$ \\
\end{tabular}

\bigskip

\noindent For a $\mu$-measurable set $E\subset X$, we denote 
\begin{align*}
&L^0(E):=\{f\colon E\to \CC : f \text{ is measurable}\}, \\
&L^1(X):=\{f\in L^0(X) : \int_{X}|f| d \mu<\infty\}, \\
&L^1_\mathrm{loc}(X):=\{f\in L^0(X) : \int_{B}|f| d \mu<\infty \text{ for all balls } B \text{ in } X\}, \\
&L^\infty(X):=\{f\in L^0(X) : f \text{ is bounded} \}, \\
&L^\infty(E):=\{f\in L^\infty(X) : f=\chi_Ef\}, \\
&L_0^\infty(E):=\{f\in L^\infty(E)\cap L^1(X) : \int_E f d \mu = 0\}, \\
&L_+^\infty(E):=\{f\in L^\infty(E) : f(x)\geq 0 \text{ for every } x\in X\}.
\end{align*}
If $f\in L^\infty(X)$, we write 
\[
\|f\|_{L^\infty(X)}:=\sup_{x\in X}|f(x)|
\]
Furthermore, we set 
\begin{align*}
&L^\infty_{\rm bs}(X):=\{f\in L^\infty(X) : \supp(f) \text{ is bounded}\} \\
&L^1_{\rm bs}(X):=\{f\in L^1(X) : \supp(f) \text{ is bounded}\}.
\end{align*}

\subsection{On quasi-metrics}
Suppose for this subsection, that we consider just a non-empty quasi-metric space $(X,d)$ (a set $X$ equipped with a quasi-metric), that does not necessarily have any extra structure required of a space of homogeneous type.

If $A$ and $B$ are non-empty subsets of $X$, we set $d(A,B):=\inf\{d(a,b) : a\in A, b\in B\}$. A subset $A\subset X$ is called \emph{bounded}, if it is included in some ball. 

The quasi-metric $d$ induces a topology of $X$ when we define: $U\subset X$ is \emph{open} if for every $x\in U$ there exists $r>0$ such that $B(x,r)\subset U$. We note that a ball of a quasi-metric space may in general fail to be open. A \emph{closure} of a subset $A\subset X$ is the closure of $A$ in the topology induced to $X$ by the quasi-metric, that is, the closure of $A$ is the set of points in $X$ whose every open neighbourhood in $X$ intersects $A$. The \emph{support} of a function $f\colon X \to \CC$ is then defined to be the closure of the set $\{x\in X : f(x)\neq 0\}$.

A result of Mac{\'\i}as and Segovia (\cite[Theorem 2]{MS1979} says that each quasi-metric $d$ on $X$ is \emph{equivalent} (refer to \cite{MS1979}) to a quasi-metric $\rho$ on $X$ which has the property that balls with respect to $\rho$ are open. (Note that the topologies induced by $d$ and $\rho$ are the same.) Shifting our focus back to the original space $(X,d)$, we get the immediate corollary that the center point $x$ of each $d$-ball $B(x,r)$ is an interior point of that ball. This fact in turn has the following natural lemmas as corollaries (that are possibly also provable ``directly'' from the definition of a quasi-metric):

\begin{lemma}\label{lemma:closureisballclosure}
Let $A\subset X$. The closure of $A$ is 
\[
\{x\in X \mid \forall r>0\colon B(x,r)\cap A\neq\emptyset\}.
\]
\end{lemma}

\begin{lemma}\label{lemma:posdistanceimpliesseparation}
Suppose $\emptyset\neq A,B\subset X$. If $d(A,B)>0$, then the closures of $A$ and $B$ are disjoint.
\end{lemma}

\begin{lemma}\label{lemma:closureofboundedisbounded}
If $A\subset X$ is bounded, then the closure of $A$ is bounded.
\end{lemma}
Note how the first lemma says in particular that even though it is possible that some balls are not open, the balls and open sets are interchangeable in some contexts.

\subsection{Geometric doubling property}
It is well-known that the doubling property of the measure for spaces of homogeneous type implies the following geometric doubling property of the quasi-metric (see e.g. \cite[p. 67]{CW1971} for a related deduction):
\begin{lemma}\label{lemma:geometricdoubling}
    Let $X$ be a space of homogeneous type. Then there exists a constant $N$ such that each ball $B(x,r)$ can be covered by at most $N$ balls of radius $r/2$.
\end{lemma}

\subsection{(Non-degenerate) $\omega$-Calder\'{o}n-Zygmund kernels on spaces of homogeneous type}

\bigskip

\begin{defi}\label{defi1.1} A function $K\colon (X \times X) \setminus \{(x,x): x \in X\} \to \CC$ is called a $\omega$-Calder\'{o}n-Zygmund kernel if it satisfies the following estimates: for all $x \neq y$,
\begin{equation}\label{eq:kernelsizeassumption}
|K(x, y)| \leq \frac{c_K}{V(x, y)}
\end{equation}
and for $d\left(x, x^{\prime}\right) < \left(2 A_0\right)^{-1} d(x, y)$,
$$
\left|K(x, y)-K\left(x^{\prime}, y\right)\right|+\left|K(y, x)-K\left(y, x^{\prime}\right)\right| \leq \frac{1}{V(x, y)} \omega\left(\frac{d\left(x, x^{\prime}\right)}{d(x, y)}\right)
$$
where $V(x, y)=\mu(B(x, d(x, y)))$ and $\omega\colon [0,1] \rightarrow[0, \infty)$ is continuous, increasing, subadditive, $\omega(0)=0$. By the doubling condition we have that $V(x, y) \approx V(y, x)$.
\end{defi}

We say that $\omega$ satisfies the Dini condition if
\begin{align}
\|\omega\|_{Dini}=\int_{0}^{1}\omega(t)\frac{dt}{t}<\infty.
\end{align}

A kernel $K$ is said to be non-degenerate, if there are positive constants $c_0$ and $\bar{C}$ such that for every $x \in X$ and $r>0$, there exists $y \in B(x, \bar{C} r) \backslash B(x, r)$ satisfying
\begin{align}\label{eq-1.1}
|K(x, y)| \geq \frac{1}{c_0 \mu(B(x, r))}.
\end{align}

\subsection{The singular integral operator on spaces of homogeneous type}

 \begin{defi}\label{defi1.2} We say that $T$ is a Calder\'{o}n-Zygmund operator on $(X, d, \mu)$ if $T$ is bounded on $L^2(X)$ and has the $\omega$-Calder\'{o}n-Zygmund kernel $K(x, y)$ such that 
 \begin{equation}\label{eq:defofsingularintegral}
 Tf(x)=\int_X K(x, y) f(y) d \mu(y)
 \end{equation}
 for any $x\notin \supp(f)$.

\end{defi}

Note that for $f\in L^2(X)$, the integral representation \eqref{eq:defofsingularintegral} is absolutely convergent for any $x\notin \supp(f)$. This follows from the kernel size assumption \eqref{eq:kernelsizeassumption}. Namely, given such an element $x$ there exists a radius $r>0$ such that $B(x,r)\subset f^{-1}(\{0\})$ (note Lemma \ref{lemma:closureisballclosure}). Then $|K(x,y)|\leq c_K V(x,y)^{-1}\leq c_K \mu(B(x,r))^{-1}=:\beta$ for all $y\notin B(x,r)$. We have
\begin{align}
    \int_X|K(x,y)f(y)| d \mu(y) &\leq \|f\|_{L^2(X)}\Big(\int_{X\setminus B(x,r)} |K(x,y)|^2 d \mu(y) \Big)^{1/2} \\
    &=\|f\|_{L^2(X)}\Big(2 \cdot \int_{0}^{\beta} t \cdot \mu(\{y\in X\setminus B(x,r) : |K(x,y)| > t\}) dt \Big)^{1/2}.
\end{align}
Given that $\beta<\infty$, the following weak-type inequality for $K$ (in a single variable) finishes the proof of the convergence of the integral:
\[
d_K(t) := \mu\big(\{y\in X\setminus\{x\} : |K(x,y)|>t \}\big)\lesssim_{c_K,C_\mu} t^{-1} \quad \text{for } t>0.
\]
The general idea for a proof of this inequality is contained in the proof of \cite[Theorem 3]{MS1979}. (For the proof, we may assume that $d_K(t)>0$.) However, a simplification of that idea suffices for our purposes. Namely: Because $|K(x,y)|>t$ implies that $\mu(B(x,d(x,y)))=V(x,y) < c_K/t$ and because $V(x,y)\to \mu(X)=\infty$ as $d(x,y)\to \infty$, it must hold that the set of such $y$ is included in a ball centred at $x$ with some radius. If $R=R(c_K,t)$ is the infimum of such radii, then after choosing some $y(t)$ relatively close to the boundary of this smallest ball (say $d(x,y(t)) > R/2$; compare with \cite{MS1979}), it follows that all $y$ in consideration are contained in a ball centred at $x$ with radius $\approx d(x,y(t))$. The doubling property finishes the proof because
\[
\mu(B(x,cd(x,y(t))))\lesssim_{C_\mu} V(x,y(t))\lesssim_{c_K}|K(x,y(t))|^{-1}<\frac{1}{t}.
\]
In fact, with similar components, one can prove that for $f\in \cup_{1\leq p<\infty}L^p(X)$ and $x\notin \supp(f)$, it holds that
\begin{equation}\label{eq:regularityofkernel}
    \int_{X}|K(x,y)f(y)| d \mu(y)<\infty,
\end{equation}
as a consequence of the kernel size assumption.

\subsection{Commutators}\label{commutsdefinition}
The space $L^\infty_{\rm bs}(X)$ of bounded functions that have bounded support is dense in $L^p(w)$ for any non-negative $w\in L^1_\mathrm{loc}(X)$ and $1<p<\infty$.

Let us note that a Calder\'{o}n-Zygmund operator $T\colon L^2(X)\to L^2(X)$ with $\|\omega\|_{Dini}<\infty$ has a unique bounded continuation $T\colon L^1(X)\to L^{1,\infty}(X)$ that satisfies
\[
\|T\|_{L^1(x)\to L^{1,\infty}(X)} \lesssim_{C_\mu,A_0} (\|T\|_{L^1(x)\to L^{1,\infty}(X)}+\|\omega\|_{Dini}).
\]
Also, for $f\in L^1(X)$ and $x\notin \supp(f)$, it holds that
\[
Tf(x) = \int_X K(x,y) f(y) d \mu(y),
\]
where $K$ is the kernel associated with $T\colon L^2(X)\to L^2(X)$.

A few words about the proof of the existence of this continuation: Because the measurable sets in $X$ are exactly the Borel sets, the standard proof with a Whitney decomposition of the level set of the maximal function combined with the Calder\'{o}n-Zygmund decomposition of a function gives the proof. We refer to the book \cite[Chapter I]{Ste1993}. However, a space of homogeneous type equipped with a quasi-metric with some non-open balls do not directly fall under these methods. (For example: Is the level set of the maximal function open then?) Anyway, one circumvents this problem, for example, by proving the weak-type $(1,1)$ of $T$ first when the quasi-metric has open balls and then using the regularization result of \cite[Theorem 2]{MS1979} to obtain the weak-type $(1,1)$ of $T$ for all spaces. Here it is noteworthy that if $T$ is a Calder\'{o}n-Zygmund operator with $\|\omega\|_{Dini}<\infty$ in a given space of homogeneous type, then substituting the quasi-metric with an equivalent one with open balls preserves this property of $T$ in the resulting new space of homogeneous type.

For $f\in L^\infty_{\rm bs}(X)$ and $b\in L^1_\mathrm{loc}(X)$, it holds that $bf\in L^1(X)$. Therefore we make the following definition of the commutator $[b,T]$ in the dense set $L^\infty_{\rm bs}(X)$:

\begin{defi}\label{defi:czcommutator}
Suppose $T$ is a Calder\'{o}n-Zygmund operator on $(X,d,\mu)$ with $\|\omega\|_{Dini}<\infty$. Let $b\in L^1_\mathrm{loc}(X)$. We define the \emph{commutator} $[b,T]$ on $f\in L^\infty_{\rm bs}(X)$ by setting
\[
[b,T]f := bTf-T(bf).
\]
\end{defi}

\subsection{A System of Dyadic Cubes}
In a space of homogeneous type  $(X, d, \mu)$, a countable family $\mathscr{D}:=\cup_{k \in \mathbb{Z}} \mathscr{D}_k, \mathscr{D}_k:=\left\{Q_\alpha^k: \alpha \in \mathscr{A}_k\right\}$, of Borel sets $Q_\alpha^k \subseteq X$ is called a system of dyadic cubes with parameters $\delta \in(0,1)$ and $0<a_1 \leq A_1<\infty$ if it has the following properties:
$$
X=\bigcup_{\alpha \in \mathscr{A}_k} Q_\alpha^k \quad \text { (disjoint union) for all } k \in \mathbb{Z} \text {; }
$$
if $\ell \geq k$, then either $Q_\beta^{\ell} \subseteq Q_\alpha^k$ or $Q_\alpha^k \cap Q_\beta^{\ell}=\emptyset$;
for each $(k, \alpha)$ and each $\ell \leq k$, there exists a unique $\beta$ such that $Q_\alpha^k \subseteq Q_\beta^{\ell}$;
for each $(k, \alpha)$ there exists at most $M$ (a fixed geometric constant) $\beta$ such that
$$
\begin{array}{r}
Q_\beta^{k+1} \subseteq Q_\alpha^k, \text { and } Q_\alpha^k=\underset{Q \in \mathscr{D}_{k+1}, Q \subseteq Q_\alpha^k}{\bigcup} Q ; \\
B\left(x_\alpha^k, a_1 \delta^k\right) \subseteq Q_\alpha^k \subseteq B\left(x_\alpha^k, A_1 \delta^k\right)=: B\left(Q_\alpha^k\right) ; \\
\text { if } \ell \geq k \text { and } Q_\beta^{\ell} \subseteq Q_\alpha^k \text {, then } B\left(Q_\beta^{\ell}\right) \subseteq B\left(Q_\alpha^k\right) .
\end{array}
$$
The set $Q_\alpha^k$ is called a dyadic cube of generation $k$ with centre point $x_\alpha^k \in Q_\alpha^k$ and sidelength $\delta^k$.
From the properties of the dyadic system above and from the doubling measure, we can deduce that there exists a constant $C_{\mu, 0}$ depending only on $C_\mu$ as in (1.2) and $a_1, A_1$ as above, such that for any $Q_\alpha^k$ and $Q_\beta^{k+1}$ with $Q_\beta^{k+1} \subset Q_\alpha^k$,
$$
\mu\left(Q_\beta^{k+1}\right) \leq \mu\left(Q_\alpha^k\right) \leq C_{\mu, 0} \mu\left(Q_\beta^{k+1}\right) .
$$

\subsection{Sparse Operators on Spaces of Homogeneous Type}
 Let $\mathscr {D}$ be a system of dyadic cubes on $X$. Given $0 < \eta< 1$, a collection $\mathscr{S} \subset \mathscr {D}$ of dyadic cubes is said to be $\eta$-sparse
provided that for every $Q \in \mathscr{S}$, there is a measurable subset $E_Q \subset Q$ such that $\mu(E_Q) \ge \eta\mu(Q)$
and the sets $\{E_Q\}_{Q\in \mathscr{S}}$ have only finite overlap.
	
 \begin{defi}\label{defi2.1} 
 Given $0<\eta<1$ and an $\eta$-sparse family $\mathscr{S}\subset \mathscr{D}$ of dyadic cubes. The sparse operator $\mathcal{A}_{\mathscr{S}}f(x)$ is defined by 
 \begin{align}  \mathcal{A}_{\mathscr{S}}f(x):=\sum_{Q\in\mathscr{S}}\langle f\rangle_{Q}\chi_{Q}(x).
 \end{align}
At the same time, the two weight fractional sparse operator is defined by 
$$\mathcal{A}_{\lambda_1, \lambda_2}^{p, q}(f ; \mathscr{S})=\sum_{P \in \mathscr{S}} \frac{\lambda_1^p(P)^{\frac{1}{p}} \lambda_2^{-q^{\prime}}(P)^{\frac{1}{q^{\prime}}}}{\mu(P)}\langle f\rangle_P \chi_P$$
 
\end{defi}

\subsection{Muckenhoupt weights}

\begin{defi}\label{defi:apweightsmuckenhoupt}
Let $w\colon X \to (0,\infty)$ be a locally integrable function and $1<p<\infty$.

We say that $w$ is an $A_p$ weight, written $w\in A_p$, if 
\[
[w]_{A_p}:=\sup_B \frac{\big(\int_B w d \mu\big) \big(\int_B w^{-\frac{1}{p-1}} d \mu \big)^{p-1}}{\mu(B)^p} < \infty,
\]
where the supremum is taken over all balls $B\subset X$.

We say that $w$ is an $A_{p,p}$ weight, written $w\in A_{p,p}$, if $w^p\in A_p$. Then we write $[w]_{A_{p,p}}:=[w^p]_{A_p}^{1/p}$.
\end{defi}

It is a direct consequence of the above definition of an $A_p$ weight $w$ and H\"{o}lder's inequality that whenever $B$ is a ball and $E\subset B$ satisfies $\mu(E)>0$, then
\[
w(B) \leq [w]_{A_p}\bigg(\frac{\mu(B)}{\mu(E)}\bigg)^p w(E)
\]
In particular, if $x\in X$, $r>0$, $\lambda > 1$, then
\[
w(B(x,\lambda r)) \leq [w]_{A_p}\bigg(\frac{\mu(B(x,\lambda r)}{\mu(B(x,r))}\bigg)^p w(B(x,r)) \leq [w]_{A_p} C_\mu^p \lambda^{\QQ p} w(B(x,r)).
\]
Thus the doubling of $\mu$ induces doubling of any $A_p$ weight $w$ in the above sense.

\subsection{On the fractional Bloom weight}
\begin{defi}\label{defi1.3}
Given two weights $\lambda_1, \lambda_2$ and exponents $1<p, q<\infty$, we define the Bloom weight
$$
\nu=\nu_{p, q}=(\lambda_1 / \lambda_2)^{\frac{1}{1 / p+1 / q^{\prime}}}=(\lambda_1 / \lambda_2)^{\frac{1}{1+\alpha / \mathbb Q}}, \quad \text{ where} \quad \frac{\alpha}{\mathbb Q}=\frac{1}{p}-\frac{1}{q}.
$$
\end{defi}

\subsection{The space of functions of weighted fractional bounded mean oscillation $\BMO^{\alpha}_{w}$}

\begin{defi}
For a positive weight $w$ and a parameter $\alpha\in \mathbb{R}$,  a function $b\in L^1_{\mathrm{loc}}(X)$ belongs to $\BMO_{w}^{\alpha}(X)$ if
\begin{align}
\|b\|_{\BMO_{w}^{\alpha}(X)}:=\sup_{B} \mathcal{M}_{w}^{\alpha}(b,B):=\sup_{B}\frac{1}{w(B)^{\frac{\alpha}{\mathbb Q}}}\bigg(\frac{1}{w(B)}\int_{B}|b(x)-\langle b \rangle_{B}|d\mu(x) \bigg) < \infty,
\end{align}
where $\langle b \rangle_{B}:=\frac{1}{\mu(B)}\int_{B}b(x)d\mu(x)$ and the supremum is taken over all balls $B\subset X$.
\end{defi}

\begin{defi}\label{def2.2}
Let $0<\eta<1$. Let $\mathscr{D}$ be a dyadic system in $X$ and let $\mathscr{S} \subset \mathscr{D}$ be a $\eta$-sparse family. Let $\lambda_1, \lambda_2$ be weights and $b \in L_{loc}^1\left(X\right)$ and define
$$
\|b\|_{\BMO_{\lambda_1, \lambda_2}^{p, q}(\mathscr{S})}=\sup _{Q \in \mathscr{S}} \frac{1}{\lambda_1^p(Q)^{\frac{1}{p}} \lambda_2^{-q^{\prime}}(Q)^{\frac{1}{q'}}} \int_Q\left|b-\langle b\rangle_Q\right| d \mu
$$
\end{defi} 

According to the proof of \cite[Proposition 3.1]{HOS2023}, we have the following  relationship between $\BMO_{\lambda_1, \lambda_2}^{p, q}(\mathscr{S})$ and $\BMO_{\nu}^{\alpha}$ functions on X.

\begin{lemma}\label{lemmaS2}
Suppose that $1<p, q<\infty$ and $\lambda_1 \in A_{p, p}, \lambda_2 \in A_{q, q}$ and let $\alpha / \mathbb{Q}=1 / p-1 / q$. Let $\nu$ be the fractional Bloom weight as in Definition \ref{defi1.3}. Then, it holds that
\[
1\leq \frac{\lambda_1^p(B)^{1/p}\lambda_2^{-q'}(B)^{1/q'}}{\nu(B)^{1+\alpha/\QQ}}\leq [\lambda_1]_{A_{p,p}}[\lambda_2]_{A_{q,q}}
\]
for all balls $B\subset X$. Also, $\nu\in A_s$, where
\[
s := 1 + \frac{1/{p'}+1/q}{1/p+1/{q'}} = \frac{2}{1+\alpha/{\QQ}}.
\]
and we have $[\nu]_{A_s}^{1/p+1/{q'}}\leq [\lambda_1]_{A_{p,p}}[\lambda_2]_{A_{q,q}}$.

If additionally $0<\eta<1$, $\mathscr{D}$ is a dyadic system in $X$, $\mathscr{S} \subset \mathscr{D}$ is an $\eta$-sparse family, and $b \in L_{\mathrm{loc}}^1\left(X\right)$, then
$$
\|b\|_{\BMO_\nu^\alpha(X)} \gtrsim\|b\|_{\BMO_{\lambda_1, \lambda_2}^{p, q}(\mathscr{S})} .
$$
\end{lemma}

\section{Proof of Theorem \ref{thm1}}\label{section 3}

In order to obtain the proof of Theorem \ref{thm1}. we recall that Duong et. al. (\cite{DGKLWY2021}) have obtained the sparse pointwise domination of the commutator on a space of homogeneous type.

\begin{lemma}[\cite{DGKLWY2021}]\label{lemmaA2.1}
Let $T$ be the Calder\'on-Zygmund operator as in Definition \ref{defi1.2} with $\omega$ satisfying the Dini condition and let $b\in L^1_\mathrm{loc}(X)$. For every $f\in L^{\infty}(X)$ with bounded support, there exist $\mathcal{T}$ dyadic systems $\mathscr{D}^{t}$, $t=1,2,\ldots,\mathcal{T}$ and $\eta$-sparse families $\mathscr{S}_{t}\subset \mathscr{D}^{t}$ such that for a.e. $x\in X$,
\begin{align}
|[b,T]f(x)|&\le C\bigg(\sum_{t=1}^{\mathcal{T}}\sum_{Q\in \mathscr{S}_t}|b(x)-\langle b \rangle_Q|\bigg(\frac{1}{\mu(Q)}\int_{Q}|f(z)|d \mu(z)\bigg)\chi_{Q}(x)\\&\quad+ \sum_{t=1}^{\mathcal{T}}\sum_{Q\in \mathscr{S}_t}\bigg(\frac{1}{\mu(Q)}\int_{Q}|b(z)-\langle b\rangle_{Q}||f(z)|d\mu(z)\bigg)\chi_{Q}(x)\bigg)\nonumber\\
&:=C \sum_{t=1}^{\mathcal{T}}\left(\mathcal{A}_{b,\mathscr{S}_t}|f|(x)+\mathcal{A}^*_{b,\mathscr{S}_t}|f|(x) \right).\nonumber   
\end{align}
\end{lemma}

By Lemma \ref{lemmaA2.1}, we need to show the following two propositions.
\begin{proposition}\label{prop3.2} Suppose $b\in \BMO_\nu^\alpha(X)$.
 Let $1<p \leq q<\infty$, $\frac{\alpha}{\mathbb Q}=\frac{1}{p}-\frac{1}{q}$, $\lambda_1 \in A_{p, p}$ and $\lambda_2 \in A_{q, q}$ and let $\nu=\nu_{p, q}$ be the fractional Bloom weight as in Definition \ref{defi1.3}. Then for the operator $\mathcal{A}_{b,\mathscr{S}}$ as in Lemma \ref{lemmaA2.1}, there exists a positive constant $C$ such that
$$
\left\|\mathcal{A}_{b,\mathscr{S}}|f|\right\|_{L_{\lambda_2}^q(X)} \leq C \|b\|_{\BMO_\nu^\alpha(X)}
\|f\|_{L_{\lambda_1}^p(X)}. $$
\end{proposition}

\begin{proposition}\label{prop3.3} Suppose $b\in \BMO_\nu^\alpha(X)$.
 Let $1<p \leq q<\infty$, $\frac{\alpha}{\mathbb Q}=\frac{1}{p}-\frac{1}{q}$, $\lambda_1 \in A_{p, p}$ and $\lambda_2 \in A_{q, q}$ and let $\nu=\nu_{p, q}$ be the fractional Bloom weight as in Definition \ref{defi1.3}. Then for the operator $\mathcal{A}^*_{b,\mathscr{S}}$
as in Lemma \ref{lemmaA2.1}, there exists a positive constant $C$ such that
$$
\left\|\mathcal{A}^*_{b,\mathscr{S}}|f|\right\|_{ L_{\lambda_2}^q(X)} \leq C \|b\|_{\BMO_\nu^\alpha(X)}
\|f\|_{L_{\lambda_1}^p(X) }.$$
\end{proposition}

In order to obtain the above two propositions, we need the following lemma.
\begin{lemma}[\cite{DGKLWY2021}]\label{LemmaA2.2}
Let $0<\eta<1$. Let $\mathscr{D}$ be a dyadic system in $X$ and let $\mathscr{S} \subset \mathscr{D}$ be a $\eta$-sparse family. Assume that $b\in L^1_\mathrm{loc}(X)$. Then there exists a $ \frac{\eta}{2(\eta+1)}$-sparse family $\tilde{\mathscr{S}} \subset \mathscr{D}$ such that $\mathscr{S} \subset \tilde{\mathscr{S}}$ and for every cube $Q \in \tilde{\mathscr{S}}$,
$$
\left|b(x)-\langle b \rangle_Q\right| \leq C \sum_{P \in \tilde{\mathscr{S}}, P \subset Q} \Omega(b, P) \chi_P(x),
$$
for a.e. $x \in Q$, where $$\Omega(b,P):=\frac{1}{\mu(P)}\int_{P}|b(x)-\langle b \rangle_P|d\mu(x).$$
\end{lemma}

  By Lemma \ref{LemmaA2.2} and Lemma \ref{lemmaS2}, we have
$$
\begin{aligned}
\left\langle\left|b-\langle b\rangle_Q\right||f|\right\rangle_Q & \lesssim\Bigg\langle\sum_{\substack{P \in \tilde{\mathscr{S}} \\
P \subset Q}}\left\langle\left|b-\langle b\rangle_P\right|\right\rangle_P|f| \chi_P\Bigg\rangle_Q=\Bigg\langle\sum_{\substack{P \in \tilde{\mathscr{S}} \\
P \subset Q}}\left\langle\left|b-\langle b\rangle_P\right|\right\rangle_P\langle|f|\rangle_P \chi_P\Bigg\rangle_Q \\
& \lesssim C_{[\lambda_1]_{A_{p, p}},[\lambda_2]_{A_{q, q}}}\|b\|_{\BMO_\nu^\alpha}\Bigg\langle\sum_{\substack{P \in \tilde{\mathscr{S}} \\
P \subset Q}} \frac{\lambda_1^p(P)^{\frac{1}{p}} \lambda_2^{-q^{\prime}}(P)^{\frac{1}{q^{\prime}}}}{\mu(P)}\langle|f|\rangle_P \chi_P\Bigg\rangle_Q
\end{aligned}.
$$
Thus, it is enough to show Proposition \ref{prop3.3}. By duality and H\"{o}lder's inequality, we have
$$
\begin{aligned}
\left\|\mathcal{A}_{\lambda_1, \lambda_2}^{p, q}(f ; \tilde{\mathscr{S}})\right\|_{L_{\lambda_2}^q(X)}&=\sup_{g:\|g\|_{L^{q^{\prime}}_{{\lambda_2}^{-1}}(X)}=1}\left|\left\langle\mathcal{A}_{\lambda_1, \lambda_2}^{p, q}(f ; \tilde{\mathscr{S}}), g\right\rangle\right|\\ & \leq \sup_{g:\|g\|_{L^{q^{\prime}}_{\lambda_2^{-1}}(X)}=1}\sum_{Q \in \tilde{\mathscr{S}}} {\lambda_1}^p(Q)^{\frac{1}{p}}\langle|f|\rangle_Q \lambda_2^{-q^{\prime}}(Q)^{\frac{1}{q^{\prime}}}\langle|g|\rangle_Q \\
& \leq\sup_{g:\|g\|_{L^{q^{\prime}}_{\lambda_2^{-1}}(X)}=1}\left(\sum_{Q \in \tilde{\mathscr{S}}}\langle|f|\rangle_Q^q {\lambda_1}^p(Q)^{\frac{q}{p}}\right)^{\frac{1}{q}}\left(\sum_{Q \in \tilde{\mathscr{S}}}\langle|g|\rangle_Q^{q^{\prime}} \lambda_2^{-q^{\prime}}(Q)\right)^{\frac{1}{q^{\prime}}} \\
& \leq\sup_{g:\|g\|_{L^{q^{\prime}}_{\lambda_2^{-1}}(X)}=1}\left(\sum_{Q \in \tilde{\mathscr{S}}}\langle|f|\rangle_Q^p {\lambda_1}^p(Q)\right)^{\frac{1}{p}}\left(\sum_{Q \in \tilde{\mathscr{S}}}\langle|g|\rangle_Q^{q^{\prime}} \lambda_2^{-q^{\prime}}(Q)\right)^{\frac{1}{q^{\prime}}(X)} \\
& \lesssim\sup_{g:\|g\|_{L^{q^{\prime}}_{\lambda_2^{-1}}(X)}=1}[\lambda_1]_{A_{p, p}}^{p^{\prime}}\left[\lambda_2^{-1}\right]_{A_{q^{\prime}, q^{\prime}}}^q\|f\|_{L_{\lambda_1}^p(X)}\|g\|_{L_{\lambda_2^{-1}}^{q^{\prime}}(X)}=[\lambda_1]_{A_{p, p}}^{p^{\prime}}[\lambda_2]_{A_{q, q}}^q\|f\|_{L_{\lambda_1}^p(X)},
\end{aligned}
$$
where we use  $\|\cdot\|_{\ell^q} \leq\|\cdot\|_{\ell^p}$ (by $p \leq q$ ).

Hence, by the above inequality, we have
\begin{align*}
    \left\|\mathcal{A}_{b, \mathscr{S}}^*f\right\|_{L_{\lambda_2}^q(X)} &\lesssim C_{[\lambda_1]_{A_{p, p}},[\lambda_2]_{A_{q, q}}}\|b\|_{\BMO_\nu^\alpha(X)}\left\|\mathcal{A}_{\lambda_1, \lambda_2}^{p, q}(f ; \tilde{\mathscr{S}})\right\|_{L_{\lambda_2}^q(X)}\\&\lesssim C_{[\lambda_1]_{A_{p, p}},[\lambda_2]_{A_{q, q}}}\|b\|_{\BMO_\nu^\alpha(X)}\|f\|_{L_{\lambda_1}^p(X) }.
\end{align*}
The proof of Theorem \ref{thm1} is complete.

\section{Proof of Theorem \ref{thm2}}\label{section 4}

\subsection{A proof using the median method}

If $b$ is a real-valued measurable function, we can use the approach and method in \cite{DGKLWY2021} (``median method''). We recall the following kernel condition from \cite{DGKLWY2021}.

\begin{defi}\label{defi4,1}
    There exist positive constants $3\le A_1\le A_2$ such that for any ball $B:=B(x_0,r)\subset X$, there exists a ball $\widetilde{B}:=B(y_0,r)$ such that $A_1r\le d(x_0,y_0)\le A_2r$, and for all $(x,y)\in B\times \widetilde{B}$, $K(x,y)$ does not change sign and 
    \begin{align}\label{eq-4.2}
        |K(x,y)|\gtrsim \frac{1}{\mu(B)}.
    \end{align}
\end{defi}

\begin{lemma}[\cite{DGKLWY2021}]
    Let $T$ be a Calder\'{o}n-Zygmund operator with kernel $K$ as Definition \ref{defi1.1}, and satisfy the ``non-degenerate" condition \eqref{eq-1.1}. Then T satisfies \eqref{eq-4.2}.
\end{lemma}

\begin{defi}
    By a median value of a real-valued measurable function $f$ over $B$, we mean a possibly non-unique, real number $\alpha_{B}(f)$ such that 
    \begin{align*}
        \mu({x\in B:f(x)>\alpha_{B}(f)})\leq \frac{1}{2}\mu(B)
    \end{align*}and 
    \begin{align*}
          \mu({x\in B:f(x)<\alpha_{B}(f)})\leq \frac{1}{2}\mu(B).
    \end{align*}
\end{defi}

\begin{lemma}[\cite{DGKLWY2021}] Let $b$ be a real-valued measurable function. For any ball $B$, let $\widetilde{B}$ be as in Definition \ref{defi4,1}. Then there exist measurable sets $E_1,E_2\subset B$
and $F_1,F_2\subset \widetilde{B}$, such that
\begin{enumerate}
    \item $B=E_1\cup E_2, \widetilde{B}=F_1 \cup F_2 $ and $\mu(F_i) \geq \frac{1}{2}\mu( \widetilde{B}),i=1,2;$
    \item $b(x)-b(y)$ does not change sign for all $(x,y)$ in $E_i \times F_i,i=1,2;$
    \item $|b(x)-\alpha_{\widetilde{B}}(b)|\leq |b(x)-b(y)|$ for all $(x,y)$ in $E_i \times F_i ,i=1,2.$
\end{enumerate}

\end{lemma}
We now return to the proof of Theorem \ref{thm2}. We set 
$$F_1:=\{y\in\widetilde{B}:b(y)\leq \alpha_{\widetilde{B}}(b) \}\quad\text{and} \quad F_2:=\{y\in\widetilde{B}:b(y)\geq \alpha_{\widetilde{B}}(b) \}$$
$$E_1:=\{y\in B:b(y)\leq \alpha_{\widetilde{B}}(b) \}\quad\text{and} \quad E_2:=\{y\in B:b(y)\geq \alpha_{\widetilde{B}}(b) \}$$
Therefore, we have
\begin{align*}
    \Omega(b,B)&=\frac{1}{\mu(B)}\int_{B}|b(x)-\langle b \rangle_{B}|d\mu(x)\\
    &\lesssim \frac{1}{\mu(B)}\sum_{i=1}^2\int_{B}|b(x)-\alpha_{\widetilde{B}} |d\mu(x)\\
    &\lesssim \frac{1}{\mu(B)}\sum_{i=1}^2\int_{E_i}\int_{F_i}\frac{|b(x)-\alpha_{\widetilde{B}} |}{\mu(B)}d\mu(y)d\mu(x)\\
    &\lesssim \frac{1}{\mu(B)}\sum_{i=1}^2\int_{E_i}\int_{F_i}|b(x)-b(y)||K(x,y)|d\mu(y)d\mu(x)\\
     &\lesssim \frac{1}{\mu(B)}\sum_{i=1}^2\int_{B}|[b,T]f_{i}(x)|d\mu(x)\\
\end{align*}
Using H\"older inequality and $[b,T]$ is a bounded operator from $L_{\lambda_1}^p(X)$ to $L_{\lambda_2}^q(X)$, we have
\begin{align*}
\frac{1}{\mu(B)}\sum_{i=1}^2\int_{B}[b,T]f_{i}(x)d\mu(x)
&=\frac{1}{\mu(B)}\sum_{i=1}^2\int_{B}|[b,T]f_{i}(x)|\lambda_2(x)\lambda^{-1}_2(x)d\mu(x)
\\&\lesssim \frac{1}{\mu(B)}\sum_{i=1}^2\bigg[\int_{B}|[b,T]f_{i}(x)|^{q}\lambda^{q}_2(x)d\mu(x)\bigg]^{\frac{1}{q}}\bigg[\int_{B}\lambda^{-q'}_2(x)d\mu(x)\bigg]^{\frac{1}{q'}}\\
&\lesssim \|[b, T]\|_{L_{\lambda_{1}}^p \rightarrow L_{\lambda_{2}}^q}\left(\lambda_{1}^p(\widetilde{B})\right)^{\frac{1}{p}}\left(\lambda_{2}^{-q^{\prime}}(B)\right)^{\frac{1}{q^{\prime}}} .
\end{align*}
By the definition of $A_{p, p}$ weights, and the doubling property, we have
$$
\lambda_{1}^p(\widetilde{B}) \lesssim[\lambda_{1}]_{A_{p, p}} \lambda_{1}^p(B).
$$
Thus, if $b$ is a real-valued measurable function. The proof of Theorem \ref{thm2} is complete.

\subsection{A proof using approximate weak factorisation}

The method of approximate weak factorisation, which was originally introduced in the Euclidean space by Hyt\"{o}nen (\cite{Hyt2021}), readily allows complex-valued multiplier functions $b$ of the commutator $[b,T]$, and not only real-valued $b$ like the median method allows. We present a proof based on the approximate weak factorisation.

To improve comparability of the present argument to the argument of \cite{Hyt2021}, we work in this subsection with the following non-degeneracy condition of the kernel $K$, which is an opposite version of \eqref{eq-1.1} in a sense.

\begin{defi}[Another non-degeneracy condition]
Suppose $K$ is an $\omega$-Calder\'{o}n-Zygmund kernel. There are positive constants $c_0$ and $\bar{C}$ such that for every $y \in X$ and $r>0$, there exists $x \in B(y, \bar{C} r) \backslash B(y, r)$ satisfying
\begin{align}\label{eq-1.1opp}
|K(x, y)| \geq \frac{1}{c_0 \mu(B(y, r))}.
\end{align}
\end{defi}
At the time of writing this paper, there already existed two forms of the non-degeneracy condition in the literature: one where the existential quantifier acts on $x$ and one where it acts on $y$. The former is how non-degeneracy was originally formulated in the Euclidean setting \cite{Hyt2021}; the formulation in spaces of homogeneous type is in \eqref{eq-1.1opp}. The latter is what we call non-degeneracy in this paper \eqref{eq-1.1}; this condition has been used at least in \cite{DGKLWY2021, LL2022}.

As the following deduction will show in the end, the two non-degeneracy conditions lead to two different consequences of the approximate weak factorisation that are in a sense dual versions of each other. This is why the two non-degeneracy conditions are interchangeable in this context, even if they were not equivalent.

We will work with assumptions on the singular integral operator $T$ that are weaker than those formulated in Theorem \ref{thm2}. The aim is to prove the following Proposition \ref{prop:oikarisawfformulationforcalderonzygmundxdmu} so that it can be then used to transform information on the commutator (boundedness) to information on the oscillation of $b$. For any unclear notation, see Section \ref{section 2} above and Definition \ref{defi:notnecessarilyboundedsio} below.

\begin{proposition}[Bounding oscillations with \eqref{eq-1.1opp}]\label{prop:oikarisawfformulationforcalderonzygmundxdmu}
Suppose $K$ is an $\omega$-Calder\'{o}n-Zygmund kernel in $X$ that satisfies \eqref{eq-1.1opp} and suppose $T\in \mathrm{SIO}(K,L^1_{\rm bs}(X))$, $b\in L^1_{\mathrm{loc}}(X)$ and $c\geq 1$. Let $B$ be a ball in $X$ with radius $r$.

Then there exists a ball $\wt{B}$ that has the same radius $r$ as $B$, lies at distance $d(B,\wt{B})\approx r$ from $B$ and the following holds: for any $E\subset B$ and $\wt{E}\subset \wt{B}$ such that $\mu(B)\leq c\mu(E)$ and $\mu(\wt{B})\leq c\mu(\wt{E})$, we have 
\begin{equation}
\int_{E}|b-\ave{b}_E| d \mu \lesssim \big|\int_{\wt{E}}g_{\wt{E}}[b,T]h_E d \mu \big|+\big|\int_{\wt{E}}h_{\wt{E}}[b,T]g_E d \mu \big|,
\end{equation}
where the auxiliary functions $g_E,h_E\in L^\infty(E)$ and $g_{\wt{E}},h_{\wt{E}}\in L^\infty(\wt{E})$ satisfy 
\[
g_E=\chi_E, \quad g_{\wt{E}}=\chi_{\wt{E}}, \quad \|h_E\|_{L^\infty(X)}\lesssim 1, \quad \|h_{\wt{E}}\|_{L^\infty(X)}\lesssim 1.
\]
The implied constants depend at most on the kernel parameters $c_0$, $\bar{C}$, $c_K$, the parameters $C_\mu$ and $A_0$ of the space $X$ and on $c$.
\end{proposition}

We begin by generalising a result of \cite{Hyt2021} for to the setting of spaces of homogeneous type. After that, we give a taste of the argument for Proposition \ref{prop:oikarisawfformulationforcalderonzygmundxdmu}.

\begin{proposition}\label{prop:propertiesofkernelsxdmu}
Let $K$ be an $\omega$-Calder\'{o}n-Zygmund kernel in $X$ that satisfies \eqref{eq-1.1opp}. Then for each $A\geq 2A_0^2+A_0$, there exists $\eps_A>0$ such that the following is true: for every ball $B=B(y_0,r)$, there is a ball $\wt{B}=B(x_0,r)$ at distance
\begin{equation}\label{eq:propertiesofkernelequaonexdmu}
d(B,\wt{B})\geq r
\end{equation}
such that: $Ar \leq d(x_0,y_0) < \bar{C}Ar$,
\begin{equation}\label{eq:propertiesofkernelequatwoxdmu}
|K(x_0,y_0)| \approx_{c_0,c_K,C_\mu,A_0} \frac{1}{\mu(B(y_0,Ar))},
\end{equation}
for all $x\in\wt{B}$ we have 
\begin{equation}\label{eq:propertiesofkernelequathreexdmuone}
\int_B |K(x,y)-K(x_0,y_0)| d \mu(y) \lesssim_{C_\mu,A_0} \eps_A\frac{\mu(B)}{\mu(B(y_0,Ar))}
\end{equation}
and for all $y\in B$ we have
\begin{equation}\label{eq:propertiesofkernelequathreexdmutwo}
\int_{\wt{B}} |K(x,y)-K(x_0,y_0)| d \mu(x) \lesssim_{C_\mu,A_0} \eps_A\frac{\mu(\wt{B})}{\mu(B(y_0,Ar))}.
\end{equation}
Additionally, $\eps_A\to0$ as $A\to\infty$.
\end{proposition}
\begin{proof}
Let $A\geq 2A_0^2+A_0$ and let $B=B(y_0,r)$ be a ball. By \eqref{eq-1.1opp}, there exists a point $x_0\in B(y_0,\bar{C}Ar)\setminus B(y_0,Ar)$ such that 
\begin{align}\label{eq:vsameasoneoverk}
\frac{1}{\mu(B(y_0,Ar))}\lesssim_{c_0} |K(x_0,y_0)|\lesssim_{c_K} \frac{1}{V(x_0,y_0)}\approx_{C_\mu,A_0} \frac{1}{V(y_0,x_0)}&=\frac{1}{\mu(B(y_0,d(x_0,y_0)))} \nonumber \\
&\leq \frac{1}{\mu(B(y_0,Ar))}.
\end{align}
Let $\wt{B}:=B(x_0,r)$. It holds that
\[
V(x_0,y_0) \approx_{c_0,c_K,C_\mu,A_0} \mu(B(y_0,Ar)) \ \text{ and } \ |K(x_0,y_0)|\approx_{c_0,c_K,C_\mu,A_0} \frac{1}{\mu(B(y_0,Ar))}.
\]
Note that by applying the quasi-triangle inequality twice, we get for all $x\in \wt{B}$ and $y\in B$ that
\begin{align*}
d(x,y)\geq A_0^{-1}d(x,y_0)-d(y,y_0)&\geq A_0^{-2}d(x_0,y_0)-A_0^{-1}d(x,x_0)-d(y,y_0) \\
&> (A_0^{-2}A-A_0^{-1}-1)r \\
&\overset{A\geq 2A_0^2+A_0}{\geq} r.
\end{align*}
Thus $d(B,\wt{B})\geq r$.

Let $x\in \wt{B}$ and $y\in B$. Then 
\begin{align}\label{eq:yyzeroxyzeroareadmissible}
d(y,y_0)<r\overset{A \geq 2A_0^2+A_0}{\leq} 2^{-1}(AA_0^{-2}r-A_0^{-1}r) &< 2^{-1}(A_0^{-2}d(x_0,y_0)-A_0^{-1}d(x,x_0)) \nonumber \\ &\leq (2A_0)^{-1}d(x,y_0)
\end{align}
and
\[
d(x,x_0)<r\overset{A>2A_0}{<} (2A_0)^{-1}Ar \leq (2A_0)^{-1}d(x_0,y_0).
\]
Hence we can write 
\begin{align*}
&|K(x,y)-K(x_0,y_0)| \\
&\leq |K(x,y)-K(x,y_0)|+|K(x,y_0)-K(x_0,y_0)| \\
&\leq \frac{1}{V(x,y_0)}\omega\big(\frac{d(y,y_0)}{d(x,y_0)}\big)+\frac{1}{V(x_0,y_0)}\omega\big(\frac{d(x,x_0)}{d(x_0,y_0)}\big) \\\\
&\leq \frac{1}{V(x,y_0)}\omega\big(\frac{r}{A_0^{-1}Ar-r}\big)+\frac{1}{V(x_0,y_0)}\omega\big(\frac{r}{Ar}\big) \\
&= \frac{1}{\mu(B(y_0,Ar))}\Big(\frac{\mu(B(y_0,Ar))}{V(x,y_0)}\omega\big(\frac{1}{A_0^{-1}A-1}\big)+\frac{\mu(B(y_0,Ar))}{V(x_0,y_0)}\omega\big(\frac{1}{A}\big)\Big).
\end{align*}
Note that if $z\in B(x_0,d(x_0,y_0))$, then
\begin{align*}
d(z,x) &\leq A_0 d(z,x_0)+A_0 d(x_0,x) \\
 &< A_0 d(x_0,y_0)+A_0r \\
&\leq A_0(A_0d(x_0,x)+A_0d(x,y_0))+A_0r \\
&< A_0^2r+A_0^2d(x,y_0)+A_0r \\
&\overset{\eqref{eq:yyzeroxyzeroareadmissible}}{<} 2^{-1}A_0d(x,y_0)+A_0^2d(x,y_0)+2^{-1}d(x,y_0) \\
&=(A_0^2+2^{-1}A_0+2^{-1})d(x,y_0)
\end{align*}
Thus 
\[
B(x_0,d(x_0,y_0))\subset B(x,(A_0^2+2^{-1}A_0+2^{-1})d(x,y_0)).
\]
Hence
\begin{equation}\label{eq:propertiesofkernelequafourxdmu}
V(x_0,y_0)\leq \mu(B(x,(A_0^2+2^{-1}A_0+2^{-1})d(x,y_0)))\leq C_\mu(A_0^2+2^{-1}A_0+2^{-1})^\mathbb{Q}V(x,y_0). 
\end{equation}
We use 
\[
\frac{\mu(B(y_0,Ar))}{V(x,y_0)} \overset{\eqref{eq:propertiesofkernelequafourxdmu}}{\lesssim_{C_\mu,A_0}} \frac{\mu(B(y_0,Ar))}{V(x_0,y_0)}\overset{\eqref{eq:vsameasoneoverk}}{\lesssim_{C_\mu,A_0}} 1.
\]
This final inequality chain allows us to continue one of our previous estimates to get
\begin{align*}
|K(x,y)-K(x_0,y_0)|&\lesssim_{C_\mu,A_0} \frac{1}{\mu(B(y_0,Ar))}\Big(\omega\big(\frac{1}{A_0^{-1}A-1}\big)+\omega\big(\frac{1}{A}\big)\Big) \\
&=:\frac{\eps_A}{\mu(B(y_0,Ar))}.
\end{align*}
It then holds that $\eps_A\to 0$ as $A\to \infty$, because $\omega$ is a modulus of continuity. Then integrating over $y\in B$ or $x\in \wt{B}$, we get for all $x_1\in \wt{B}$ and $y\in B$ that 
\begin{align*}
\int_B |K(x_1,y)-K(x_0,y_0)| d \mu(y)\lesssim_{C_\mu,A_0} \eps_A\frac{\mu(B)}{\mu(B(y_0,Ar))},
\end{align*}
and 
\begin{align*}
\int_{\wt{B}} |K(x,y_1)-K(x_0,y_0)| d \mu(x) \lesssim_{C_\mu,A_0} \eps_A\frac{\mu(\wt{B})}{\mu(B(y_0,Ar))}.
\end{align*}
\end{proof}

How to use the norm of the commutator $[b,T]$ to bound oscillations of $b$ over balls $B$? The idea, that is taken from \cite{Hyt2021}, is to use a decomposition 
\[
f=-\frac{-f}{T^*g}T^*g=:-hT^*g=gTh-hT^*g-gTh=:gTh-hT^*g+\tilde{f},
\]
where $T^*$ is in some sense the transpose of $T$ and $f$ is in the dual of the space that $b$ belongs to and supported in the ball $B$. Here one has to check that one does not divide by zero when one divides by $T^*g$. Proposition \ref{prop:propertiesofkernelsxdmu} will help there. Then we could formally write, when $f$ is chosen so that $| \langle b, f \rangle |$ approximates $\int_B |b-\ave{b}_B|$,
\begin{align*}
  \int_B |b-\ave{b}_B| \eqsim | \langle b, f \rangle | = | \langle b, gTh\rangle - \langle b, hT^*g \rangle + \langle b, \tilde{f} \rangle | &= | \langle bTh, g\rangle - \langle T(bh), g \rangle + \langle b, \tilde{f} \rangle | \\
  &= | \langle [b,T]h, g\rangle + \langle b, \tilde{f} \rangle | \\
  &\leq | \langle [b,T]h, g\rangle | + | \langle b, \tilde{f} \rangle |.
\end{align*}
This inequality chain provides information that the dual pairing for the commutator bounds oscillations of $b$, modulo an error term. The aim is to absorb the error term $| \langle b, \tilde{f} \rangle |$ to the left-hand side of this inequality chain. To make this possible, we use our decomposition once for $\tilde{f}$ so that our upper bound will in the end have a sum of two dual pairings for the commutator and a suitable error term $| \langle b, \tilde{\tilde{f}} \rangle |$. This second iteration of the decomposition will also make it so that we do not need to qualitatively assume a priori that $b$ has bounded mean oscillation: it is enough to assume local integrability of $b$.

Details will be provided. For the purposes of proving a lower bound for the commutator norm $\|[b,T]\|$, the boundedness of $T$ does not play a role. Thus we work with the assumption that $T$ is as described in Definition \ref{defi:notnecessarilyboundedsio} and additionally satisfies the condition \eqref{eq-1.1opp}.

\begin{defi}\label{defi:notnecessarilyboundedsio}
Let $K$ be an $\omega$-Calder\'{o}n-Zygmund kernel in $X$.

Suppose $\mathcal{F}\subset L^1(X)$. Suppose that the mapping $T\colon \mathcal{F}\to L^0(X)$ satisfies the following requirement: for $f\in \mathcal{F}$ and for every $x\in X\setminus \supp(f)$,
\begin{equation}\label{eq:whatsioshavetoatminimumsatisfyxdmu}
Tf(x)=\int_{X}K(x,y)f(y) d \mu(y).
\end{equation}
Then we write $T\in \mathrm{SIO}(K,\mathcal{F})$.

Note that the integral \eqref{eq:whatsioshavetoatminimumsatisfyxdmu} converges absolutely; a consequence of the upper bound for $|K|$. One should think that $T$ is a kind of singular integral mapping on its domain $\mathcal{F}$, associated with the kernel $K$.
\end{defi}

We now define the commutator for our not necessarily bounded singular integral mappings.

\begin{defi}
Suppose that $K$ is an $\omega$-Calder\'{o}n-Zygmund kernel in $X$ and suppose that $T\in \mathrm{SIO}(K,L^1_{\rm bs}(X))$. Let $b\in L^1_{\mathrm{loc}}(X)$.

We define the \emph{commutator} $[b,T]\colon L^\infty_{\rm bs}(X)\to L^0(X)$ by setting
\[
[b,T]f:=bTf-T(bf).
\]
\end{defi}

\begin{remark}
Note that the above definition of a commutator is well-set, because $f\in L^\infty_{\rm bs}(X)\subset L^1_{\rm bs}(X)$ and $bf\in L^1_{\rm bs}(X)$, when $f\in L^\infty_{\rm bs}(X)$. Also, it agrees with Definition \ref{defi:czcommutator}.
\end{remark}

For an $\omega$-Calder\'{o}n-Zygmund kernel $K$ in $X$, we write $K^*(x,y):=K(y,x)$. It is quite immediate that then $K^*$ is also an $\omega$-Calder\'{o}n-Zygmund kernel in $X$. We record this as a lemma.

\begin{lemma}\label{lemma:adjointkerneliskernel}
    Suppose $K$ is an $\omega$-Calder\'{o}n-Zygmund kernel in $X$. Then $K^*$ is also an $\omega$-Calder\'{o}n-Zygmund kernel in $X$. In particular, we have
    \begin{equation}\label{eq:adjointkernelsizeassumption}
    |K^*(x,y)|\leq \frac{C_\mu (2A_0)^{\QQ}c_K}{V(x,y)},
    \end{equation}
    where $x\neq y$ and $c_K$ is the constant for $K$ from \eqref{eq:kernelsizeassumption}.
\end{lemma}

We now present a practical lemma that essentially is just checking that we may change integration order in certain double integrals.

\begin{lemma}\label{lemma:tandtstarareadjointsxdmu}
Suppose $K$ is an $\omega$-Calder\'{o}n-Zygmund kernel in $X$. Suppose that $O$ and $P$ are balls and that $d(O,P)>0$. Suppose $f\in L^\infty(O)$ and $g\in L^\infty(P)$. Suppose $b\in L^1_\mathrm{loc}(X)$. Then
\[
\int_{P} \int_{O} b(x)K(x,y)f(y)g(x) d \mu(y) d \mu(x)=\int_{O} \int_{P} b(x)K^*(y,x)f(y)g(x) d \mu(x) d \mu(y),
\]
where the both the inner and outer integrals converge absolutely.
\end{lemma}
\begin{proof}
We first note that the inner integrals are absolutely convergent for all $x\in P$ or for all $y\in O$, respectively. This is because $f\in L^1(X)$, $bg\in L^1(X)$ and because the closures of $O$ and $P$ in $X$ are disjoint (by Lemma \ref{lemma:posdistanceimpliesseparation}).

Note that $(x,y)\mapsto b(x)K(x,y)f(y)g(x)$ is measurable $P\times O\to\CC$. We further justify the use of Fubini's theorem by proving that
\[
\int_P \int_O |b(x)K(x,y)f(y)g(x)| d \mu(y) d \mu(x)<\infty.
\]
Denote the radius of $P$ by $r$ and set $s=d(O,P)$. Let $(x,y)\in P\times O$. We have
\[
V(x,y)=\mu(B(x,d(x,y)))\geq \mu(B(x,s)).
\]
Note that $P\subset B(x,2A_0r)=B(x,cs)$, where $c:=2A_0rs^{-1}$. In case $c\leq 1$, we see that $P\subset B(x,s)$. If on the other hand $c>1$, then
\[
\mu(P)\leq \mu(B(x,cs))\leq C_\mu c^\mathbb{Q}\mu(B(x,s)).
\]
Thus in any case 
\[
V(x,y)\geq \mu(B(x,s))\gtrsim_{C_\mu,A_0,r,s} \mu(P)>0.
\]
Using this estimate, we see that 
\begin{align*}
&\int_P \int_O |b(x)K(x,y)f(y)g(x)| d \mu(y) d \mu(x) \leq \|f\|_{L^\infty}\|g\|_{L^\infty}\int_P |b(x)|\int_O |K(x,y)| d \mu(y) d \mu(x) \\
&\lesssim_{P,C_\mu,A_0}\|f\|_{L^\infty}\|g\|_{L^\infty}\frac{c_K\mu(O)}{\mu(P)}\int_P |b(x)| d \mu(x) <\infty.
\end{align*}
Thus we may exchange the integration order to get the claim.
\end{proof}

\begin{corollary}\label{cor:tandtstarareadjointsxdmu}
Suppose $K$ is an $\omega$-Calder\'{o}n-Zygmund kernel in $X$. Suppose that $O$ and $P$ are balls and that $d(O,P)>0$. Suppose $f\in L^\infty(O)$ and $g\in L^\infty(P)$. 
\begin{itemize}
\item[(i)] Suppose that $b\in L^\infty(X)$, $T\in \mathrm{SIO}(K,L^\infty_{\rm bs}(X))$ and $T^*\in \mathrm{SIO}(K^*,L^\infty_{\rm bs}(X))$. Then
\[
\int_X Tf\cdot bg d \mu=\int_X f\cdot T^*(bg) d \mu.
\]
\item[(ii)]
Suppose that $b\in L^1_\mathrm{loc}(X)$, $T\in \mathrm{SIO}(K,L^1_{\rm bs}(X))$ and $T^*\in \mathrm{SIO}(K^*,L^1_{\rm bs}(X))$. Then
\[
\int_X Tf\cdot bg d \mu=\int_X f\cdot T^*(bg) d \mu.
\]
\end{itemize}
\end{corollary}

In Corollary \ref{cor:tandtstarareadjointsxdmu} and in what follows, when $T$ is a singular integral (in the sense of Definition \ref{defi:notnecessarilyboundedsio}) having kernel $K$, the notation $T^*$ will just be used to denote a singular integral having kernel $K^*$. In particular, $T^*$ is not read here as ``the adjoint operator of $T$''. Although, from Corollary \ref{cor:tandtstarareadjointsxdmu} it is immediate that such a $T^*$ has properties similar to an adjoint operator.
\medskip

\underline{Notation:} For this section only, we use the following notation: Throughout, we suppose that $K$ is an $\omega$-Calder\'{o}n-Zygmund kernel in $X$. 
Suppose $\xi\geq 1$, $A\geq 2A_0^2+A_0$, $\eps>0$, $B=B(y_0,r)$ and $\wt{B}=B(x_0,r)$ are such that the following statements hold:
\begin{equation}\label{eq:kernelpropproof0xdmu}
d(B,\wt{B})\geq r,
\end{equation}
\begin{equation}\label{eq:kernelpropproof1xdmu}
Ar\leq d(x_0,y_0)\leq \xi Ar,
\end{equation}
\begin{equation}\label{eq:kernelpropproof2xdmu}
|K(x_0,y_0)|\leq \xi \frac{1}{\mu(B(y_0,Ar))} \quad \text{ and } \quad \frac{1}{\mu(B(y_0,Ar))}\leq \xi |K(x_0,y_0)|,
\end{equation}
for all $x\in \wt{B}$ we have
\begin{equation}\label{eq:kernelpropproof3xdmu}
\int_B |K(x,y)-K(x_0,y_0)| d \mu(y) \leq \xi \eps \frac{\mu(B)}{\mu(B(y_0,Ar))},
\end{equation}
and for all $y\in B$ we have
\begin{equation}\label{eq:kernelpropproof4xdmu}
\int_{\wt{B}} |K(x,y)-K(x_0,y_0)| d \mu(x) \leq \xi \eps \frac{\mu(\wt{B})}{\mu(B(y_0,Ar))}.
\end{equation}
For the purposes of this section, we call such a sextuple $(K,\xi,A,\eps,B,\wt{B})$ \emph{admissible}.

The motivation behind this notion is that a kernel satisfying \eqref{eq-1.1opp} is admissible by Proposition \ref{prop:propertiesofkernelsxdmu}. Additionally, unlike the condition \eqref{eq-1.1opp}, it is easy to see that admissibility of $K$ is transferred to admissibility of $K^*$ as the following proposition shows. Thus because we will use the same decomposition result for $K$ and then to $K^*$, we will formulate the decomposition result Lemma \ref{lemma:hytonenlemma2.5xdmu} with the assumption of an admissible kernel as opposed to an assumption of the kernel satisfying \eqref{eq-1.1opp}.

\begin{proposition}
Suppose $(K,\xi,A,\eps,B,\wt{B})$ is admissible. Then $(K^*,\xi^*,A,\eps,\wt{B},B)$ is admissible, where
\[
\xi^* = \xi C_\mu(A_0(1+\xi ))^\mathbb{Q}.
\]
\end{proposition}
\begin{proof}
Denote $B=B(y_0,r)$ and $\wt{B}=B(x_0,r)$. We note first that
\[
\frac{1}{C_\mu(A_0(1+\xi ))^\mathbb{Q}}\overset{\eqref{eq:kernelpropproof1xdmu}}{\leq}\frac{\mu(B(y_0,Ar))}{\mu(B(x_0,Ar))}\overset{\eqref{eq:kernelpropproof1xdmu}}{\leq}C_\mu(A_0(1+\xi ))^\mathbb{Q}.
\]
Thus we see that 
\begin{align*}
|K^*(y_0,x_0)|=|K(x_0,y_0)|\overset{\eqref{eq:kernelpropproof2xdmu}}{\leq} \frac{\xi }{\mu(B(y_0,Ar))} \leq \frac{\xi  C_\mu(A_0(1+\xi ))^\mathbb{Q}}{\mu(B(x_0,Ar))},
\end{align*}
and 
\begin{align*}
\frac{1}{\mu(B(x_0,Ar))} \leq \frac{C_\mu(A_0(1+\xi ))^\mathbb{Q}}{\mu(B(y_0,Ar))} &\overset{\eqref{eq:kernelpropproof2xdmu}}{\leq} \xi C_\mu(A_0(1+\xi ))^\mathbb{Q}|K(x_0,y_0)| \\
&= \xi C_\mu(A_0(1+\xi ))^\mathbb{Q}|K^*(y_0,x_0)|.
\end{align*}
Also, thus for all $y\in B$ we have
\begin{align*}
\int_{\wt{B}}|K^*(y,x)-K^*(y_0,x_0)| d \mu(x) &= \int_{\wt{B}}|K(x,y)-K(x_0,y_0)| d \mu(x) \\
&\overset{\eqref{eq:kernelpropproof4xdmu}}{\leq} \xi  \eps \frac{\mu(\wt{B})}{\mu(B(y_0,Ar))} \\
&\leq \xi  C_\mu(A_0(1+\xi ))^\mathbb{Q} \eps \frac{\mu(\wt{B})}{\mu(B(x_0,Ar))}.
\end{align*}
Similarly, for all $x\in \wt{B}$ we thus have
\begin{align*}
\int_{B} |K^*(y,x)-K^*(y_0,x_0)| d \mu(y) &= \int_{B} |K(x,y)-K(x_0,y_0)| d \mu(y) \\
&\overset{\eqref{eq:kernelpropproof3xdmu}}{\leq} \xi  \eps \frac{\mu(B)}{\mu(B(y_0,Ar))} \\
&\leq \xi  C_\mu(A_0(1+\xi ))^\mathbb{Q} \eps \frac{\mu(B)}{\mu(B(x_0,Ar))}.
\end{align*}
\end{proof}

We initiate the approximate weak factorisation process. The decisive difference in the proof of the factorisation, when compared to the Euclidean setting of \cite{Hyt2021}, is that due to the possible non-translativity of the measure $\mu$ we keep ourselves from estimating the coefficient $\mu(B)/\mu(\wt{B})$ further (in particular we let it depend on the ball) until the corresponding constant in the second iteration cancels this coefficient with its reciprocal. Because the distance of the center points of the balls $B$ and $\wt{B}$ is greater than $Ar$, estimating the ratio of their measures further would seem to inevitably produce a coefficient with unwanted or non-necessary dependency on the parameter $A$.

\begin{lemma}\label{lemma:hytonenlemma2.5xdmu}
Suppose that $(K,\xi ,A,\eps,B,\wt{B})$ is admissible, where $B=B(y_0,r)$ and $\wt{B}=B(x_0,r)$. Suppose that $T\in \mathrm{SIO}(K,L^\infty_{\rm bs}(X))$, $T^*\in \mathrm{SIO}(K^*,L^\infty_{\rm bs}(X))$ and $c\geq 1$. If additionally $\eps\leq 2^{-1}c^{-1}\xi ^{-2}$, then the following holds:

Suppose $f\in L_0^\infty(B)$ and $g\in L_+^\infty(\wt{B})$ is such that 
\[
0<\|g\|_{L^\infty(X)}\leq \frac{c}{\mu(\wt{B})}\int_{\wt{B}}g d \mu.
\]
Then there is a decomposition 
\[
f=gTh-hT^*g+\tilde{f},
\]
where $\tilde{f}\in L_0^\infty(\{x\in \wt{B} : g(x)\neq 0\})$ and $h\in L^\infty(\{y\in B : f(y)\neq 0\})$ satisfy 
\[
\|g\|_{L^\infty(X)}\|h\|_{L^\infty(X)}\lesssim_{c,\xi } \frac{\mu(B(y_0,Ar))}{\mu(\wt{B})}\|f\|_{L^\infty(X)}, \ \ \ \|\tilde{f}\|_{L^\infty(X)}\lesssim_{c,\xi }\eps\frac{\mu(B)}{\mu(\wt{B})}\|f\|_{L^\infty(X)}.
\]
\end{lemma}
\begin{proof}
Note that $\supp(f)$ is included in the closure of $B$ and $\supp(g)$ is included in the closure of $\wt{B}$. By \eqref{eq:kernelpropproof0xdmu} it holds that $d(B,\wt{B})>0$. Hence by Lemma \ref{lemma:posdistanceimpliesseparation} we conclude that 
\[
\wt{B}\subset X\setminus \supp(f) \quad \text{ and } \quad B\subset X\setminus \supp(g).
\]
Also, by Lemma \ref{lemma:closureofboundedisbounded} both $f$ and $g$ are in $L^\infty_{\rm bs}(X)$. 

Let $y\in B$. Then $y \notin \supp(g)$ and
\begin{align*}
T^*g(y)&=\int_{\wt{B}}K(x,y)g(x) d \mu(x) \\
&=\underbrace{K(x_0,y_0)\int_{\wt{B}}g(x) d \mu(x)}_{:=I}+\underbrace{\int_{\wt{B}}[K(x,y)-K(x_0,y_0)]g(x) d \mu(x)}_{:=II}.
\end{align*}
Using \eqref{eq:kernelpropproof2xdmu} and our assumption we see that 
\begin{align*}
|I|\geq \xi ^{-1}\frac{1}{\mu(B(y_0,Ar))}\int_{\wt{B}}g d \mu &= \xi ^{-1} \frac{\mu(\wt{B})}{\mu(B(y_0,Ar))}\frac{1}{\mu(\wt{B})}\int_{\wt{B}}g d \mu \\
&\geq c^{-1}\xi ^{-1}\frac{\mu(\wt{B})}{\mu(B(y_0,Ar))}\|g\|_{L^\infty(X)}
\end{align*}
and by \eqref{eq:kernelpropproof4xdmu} we see that
\begin{align*}
|II|\leq\|g\|_{L^\infty(X)}\int_{\wt{B}}|K(x,y)-K(x_0,y_0)| d \mu(x) &\leq \xi \eps\frac{\mu(\wt{B})}{\mu(B(y_0,Ar))}\|g\|_{L^\infty(X)}.
\end{align*}
Thus we have
\begin{align*}
|T^*g(y)|&\geq |I|-|II| \\
&\geq ( c^{-1}\xi ^{-1}-\xi \eps)\frac{\mu(\wt{B})}{\mu(B(y_0,Ar))}\|g\|_{L^\infty(X)}
\end{align*}
for all $y\in B$. We then require $0<\eps\leq 2^{-1}c^{-1}\xi ^{-2}$, as a consequence of which we continue the last estimate to get 
\begin{equation}\label{eq:smartlowerboundfortstarg}
|T^*g(y)|\geq 2^{-1}c^{-1}\xi ^{-1}\frac{\mu(\wt{B})}{\mu(B(y_0,Ar))}\|g\|_{L^\infty(X)}>0
\end{equation}
for all $y\in B$, an estimate that we shall use later.

We define $h=-\frac{f}{T^*g}$ in $B$ and $h=0$ outside $B$. We just showed that this definition does not involve division by zero. Note that $h(x)=0$ outside $B$ and also if $f(x)=0$. Then $h\in L^\infty(\{y\in B : f(y)\neq 0\})$ because from the previous considerations it follows that 
\begin{align*}
\|g\|_{L^\infty(X)}\|h\|_{L^\infty(X)}&\leq \|g\|_{L^\infty(X)}\frac{\|f\|_{L^\infty(X)}}{2^{-1}c^{-1}\xi ^{-1}\frac{\mu(\wt{B})}{\mu(B(y_0,Ar))}\|g\|_{L^\infty(X)}} \\
&\approx_{c,\xi } \frac{\mu(B(y_0,Ar))}{\mu(\wt{B})} \|f\|_{L^\infty(X)}.
\end{align*}
In particular, $h\in L^\infty_{\rm bs}(X)$ by Lemma \ref{lemma:closureofboundedisbounded}.

Define $\tilde{f}=f-gTh+hT^*g$, as we would like to have in the decomposition. Then $\tilde{f}=-gTh$, because $-f=0$ outside $B$. Also, $\tilde{f}=0$ outside $\wt{B}$ and whenever $g=0$.  We noted earlier that $\wt{B}\subset X\setminus \supp(f)$. By definition of $h$, this implies that $\wt{B}\subset X\setminus \supp(h)$. With use of \eqref{eq:propertiesofkernelequaonexdmu} and Corollary \ref{cor:tandtstarareadjointsxdmu} (choose $b\equiv 1$),
\begin{align*}
\int_{X} gTh d \mu = \int_{X} hT^*g d \mu =-\int_{X}f d \mu = 0.
\end{align*}
Thus we get $\int_{X} \tilde{f} d \mu = -\int_{X}gTh d \mu=0$. It remains to show that $\tilde{f}\in L^\infty(X)$.

We note that we can write $h=-f\cdot (f_1+f_2)$. Here $f_1=\frac{1}{T^*g}-\frac{1}{K(x_0,y_0)\int_{\wt{B}}g d \mu}$ in $B$ and $f_1= 0$ outside $B$. Also, $f_2=\frac{1}{K(x_0,y_0)\int_{\wt{B}}g d \mu}$ in $B$ and $f_2= 0$ outside $B$. Note that $f_1f=-h-f_2f\in L^\infty(X)$, because $h,f_2,f\in L^\infty(X)$. Also, $f_1f$ has its support included in the support of $f$.

Fix $x\in\wt{B}$. Recall that $\wt{B} \subset X \setminus \supp(h)$. We get 
\begin{align*}
-Th(x)&=\int_{B}K(x,y)(-h(y)) d \mu(y) \\
&=\int_B K(x,y)f_1(y)f(y) d \mu(y)+\frac{1}{K(x_0,y_0)\int_{\wt{B}}g d \mu}\int_B K(x,y)f(y) d \mu(y),
\end{align*}
where in the last row we used that $x\in X \setminus \supp(f)$. We denote 
\[
I':=\int_B K(x,y)f_1(y)f(y) d \mu(y)
\]
and 
\[
II':=\frac{1}{K(x_0,y_0)\int_{\wt{B}}g d \mu}\int_B K(x,y)f(y) d \mu(y).
\]
For $y\in B$, 
\begin{align*}
&|f_1(y)| \\
&=\bigg|\frac{K(x_0,y_0)\int_{\wt{B}}g d \mu-T^*g(y)}{T^*g(y)K(x_0,y_0)\int_{\wt{B}}g d \mu}\bigg| \\
&\leq \frac{1}{|T^*g(y)K(x_0,y_0)\int_{\wt{B}}g d \mu |}\int_{\wt{B}}|g(x_1)||K(x_0,y_0)-K(x_1,y)| d \mu(x_1) \\ 
&\overset{\eqref{eq:kernelpropproof4xdmu}}{\lesssim_\xi } \frac{1}{|T^*g(y)K(x_0,y_0)\int_{\wt{B}}g d \mu |}\|g\|_{L^\infty(X)}\eps \frac{\mu(\wt{B})}{\mu(B(y_0,Ar))} \\
&\overset{\eqref{eq:kernelpropproof2xdmu},\eqref{eq:smartlowerboundfortstarg}}{\lesssim_{c,\xi }} \eps \frac{\mu(B(y_0,Ar))^2}{\mu(\wt{B})^2}\frac{1}{\|g\|_{L^\infty(X)}^2} \|g\|_{L^\infty(X)} \frac{\mu(\wt{B})}{\mu(B(y_0,Ar))} \\
&=\eps \frac{\mu(B(y_0,Ar))}{\mu(\wt{B})\|g\|_{L^\infty(X)}}
\end{align*}
Hence 
\begin{align*}
|I'|&\lesssim_{c,\xi } \eps \frac{\mu(B(y_0,Ar))}{\mu(\wt{B})\|g\|_{L^\infty(X)}} \int_{B}|K(x,y)||f(y)| d \mu(y) \\
&\leq \eps \frac{\mu(B(y_0,Ar))}{\mu(\wt{B})}\frac{\|f\|_{L^\infty(X)}}{\|g\|_{L^\infty(X)}} \int_{B}|K(x,y)| d \mu(y),
\end{align*}
where 
\begin{align*}
\int_{B}|K(x,y)| d \mu(y)&\leq \int_B |K(x_0,y_0)| d \mu(y) + \int_B |K(x,y)-K(x_0,y_0)| d \mu(y) \\
&\overset{\eqref{eq:kernelpropproof2xdmu},\eqref{eq:kernelpropproof3xdmu}}{\lesssim_\xi } \frac{\mu(B)}{\mu(B(y_0,Ar))} + \eps \frac{\mu(B)}{\mu(B(y_0,Ar))} \\
&\lesssim \frac{\mu(B)}{\mu(B(y_0,Ar))},
\end{align*}
where in the last step we used our requirement for $\eps$ (or we could assume $\eps \leq 1$ anyway because in the end we are concerned with sufficiently small $\eps$). Substituting this estimate back, we get that 
\[
|I'|\lesssim_{c,\xi } \eps \frac{\mu(B)}{\mu(\wt{B})}\frac{\|f\|_{L^\infty(X)}}{\|g\|_{L^\infty(X)}}.
\]
Recalling that $f\in L_0^\infty(B)$, 
\begin{align*}
\big|\int_B K(x,y)f(y) d \mu(y)\big|&=\big|\int_B (K(x,y)-K(x_0,y_0))f(y) d \mu(y)\big| \\
&\leq \|f\|_{L^\infty(X)}\int_B |K(x,y)-K(x_0,y_0)| d \mu(y) \\
&\overset{\eqref{eq:kernelpropproof3xdmu}}{\lesssim_\xi } \eps \frac{\mu(B)}{\mu(B(y_0,Ar))}\|f\|_{L^\infty(X)}
\end{align*}
and thus 
\begin{align*}
|II'|&\lesssim_\xi  \eps \frac{1}{|K(x_0,y_0)\int_{\wt{B}}g d \mu |}\frac{\mu(B)}{\mu(B(y_0,Ar))}\|f\|_{L^\infty(X)} \\
&\overset{\eqref{eq:kernelpropproof2xdmu}}{\lesssim_{c,\xi }} \eps \frac{\mu(B)}{\mu(\wt{B})}\frac{\|f\|_{L^\infty(X)}}{\|g\|_{L^\infty(X)}} .
\end{align*}
Thus immediately 
\begin{align*}
\|-gTh\|_{L^\infty(X)}&\lesssim_{c,\xi } \|g\|_{L^\infty(X)} \eps \frac{\mu(B)}{\mu(\wt{B})}\frac{\|f\|_{L^\infty(X)}}{\|g\|_{L^\infty(X)}} \\
&= \eps \frac{\mu(B)}{\mu(\wt{B})} \|f\|_{L^\infty(X)}.
\end{align*}
The proof is complete.
\end{proof}

We iterate the previous lemma (but just once more for the kernel's transpose) to get the useful property that the error term is supported on the same set as the original function.

\begin{lemma}\label{lemma:hytonenlemma2.6firstversionxdmu}
Suppose that $(K,\xi ,A,\eps,B,\wt{B})$ is admissible, $T\in \mathrm{SIO}(K,L^\infty_{\rm bs}(X))$, $T^*\in \mathrm{SIO}(K^*,L^\infty_{\rm bs}(X))$ and $c\geq 1$. There is a constant $u=u(c,\xi ,C_\mu,A_0)>0$ such that if additionally $\eps\leq u$, then the following holds:

Suppose $E\subset B$ and $\wt{E}\subset \wt{B}$ are such that $\mu(B)\leq c\mu(E)$ and $\mu(\wt{B})\leq c\mu(\wt{E})$. If $f\in L_0^\infty(E)$, there is a decomposition 
\[
f=\sum_{i=1}^2(g_iTh_i-h_iT^*g_i)+\tilde{\tilde{f}},
\]
where $\tilde{\tilde{f}}\in L_0^\infty(E)$, $g_i\in L^\infty(\wt{E})$ and $h_i\in L^\infty(E)$ satisfy 
\begin{equation}\label{eq:lemma2.6xdmu}
\|g_i\|_{L^\infty(X)}\|h_i\|_{L^\infty(X)}\lesssim_{c,\xi ,C_\mu,A_0} A^\mathbb{Q} \|f\|_{L^\infty(X)}, \ \ \ \ \ \ \|\tilde{\tilde{f}}\|_{L^\infty(X)}\lesssim_{c,\xi ,C_\mu,A_0}\eps\|f\|_{L^\infty(X)}.
\end{equation}
Additionally, $g_1=\chi_{\wt{E}}$ and $h_2=\chi_E$.
\end{lemma}
\begin{proof}
Suppose $B=B(y_0,r)$ and $\wt{B}=B(x_0,r)$. Suppose 
\[
\eps \leq \min\{2^{-1}c^{-1}\xi ^{-2}, 2^{-1}c^{-1}(\xi^*)^{-2}\} \quad \big(=2^{-1}c^{-1}(\xi^*)^{-2}\big),
\]
where $\xi^*:=\xi C_\mu(A_0(1+\xi))^\mathbb{Q}$.

We first apply Lemma \ref{lemma:hytonenlemma2.5xdmu} to $(K,\xi,A,\eps,B,\wt{B})$, $c$ and the functions $f$ and $g_1:=\chi_{\wt{E}}$. It yields the decomposition 
\[
f=g_1Th_1-h_1T^*g_1+\tilde{f},
\]
where $\tilde{f}\in L_0^\infty(\wt{E})$ and $h_1\in L^\infty(\{y\in B : f(y)\neq 0\})$ satisfy 
\[
\|h_1\|_{L^\infty(X)}\lesssim_{c,\xi} \frac{\mu(B(y_0,Ar))}{\mu(\wt{B})}\|f\|_{L^\infty(X)}
\] 
and
\[
\|\tilde{f}\|_{L^\infty(X)}\lesssim_{c,\xi}\eps\frac{\mu(B)}{\mu(\wt{B})}\|f\|_{L^\infty(X)}.
\]
We may further continue the above estimate for $h_1$ by noting that 
\[
\frac{\mu(B(y_0,Ar))}{\mu(\wt{B})}\lesssim_{C_\mu}A^\mathbb{Q}\frac{\mu(B(y_0,Ar))}{\mu(B(x_0,Ar))}\overset{\eqref{eq:kernelpropproof1xdmu}}{\lesssim_{\xi,C_\mu,A_0}}A^\mathbb{Q}.
\]
Thus 
\[
\|h_1\|_{L^\infty(X)}\lesssim_{c,\xi,C_\mu,A_0}A^\mathbb{Q}\|f\|_{L^\infty(X)}.
\]

We then wish to apply Lemma \ref{lemma:hytonenlemma2.5xdmu} again, but this time to the admissible sextuple 
\[
(K^*,\xi C_\mu(A_0(1+\xi))^\mathbb{Q},A,\eps, \wt{B},B)
\]
and $c$. To this end, consider $\tilde{f}$ and $\tilde{g}:=\chi_E$. Note that because $(K^*)^*=K$, $T^*\in \mathrm{SIO}(K^*,L^\infty_{\rm bs}(X))$ and $T\in \mathrm{SIO}((K^*)^*,L^\infty_{\rm bs}(X))$,  This yields the decomposition 
\[
\tilde{f}=\tilde{g}T^*\tilde{h}-\tilde{h}T\tilde{g}+\tilde{\tilde{f}},
\]
where $\tilde{\tilde{f}}\in L_0^\infty(E)$ and $\tilde{h}\in L^\infty(\{x\in \wt{B} : \tilde{f}(x)\neq 0\})$ satisfy 
\begin{align*}
\|\tilde{h}\|_{L^\infty(X)}\lesssim_{c,\xi,C_\mu,A_0} \frac{\mu(B(x_0,Ar))}{\mu(B)}\|\tilde{f}\|_{L^\infty(X)} &\lesssim_{c,\xi}  \frac{\mu(B(x_0,Ar))}{\mu(B)}\frac{\mu(B)}{\mu(\wt{B})}\|f\|_{L^\infty(X)} \\
&= \frac{\mu(B(x_0,Ar))}{\mu(\wt{B})}\|f\|_{L^\infty(X)} \\
&\lesssim_{C_\mu} A^\mathbb{Q} \|f\|_{L^\infty(X)}
\end{align*}
and
\[
\|\tilde{\tilde{f}}\|_{L^\infty(X)}\lesssim_{c,\xi,C_\mu,A_0}\eps\frac{\mu(\wt{B})}{\mu(B)}\|\tilde{f}\|_{L^\infty(X)}\lesssim_{c,\xi}\eps\frac{\mu(\wt{B})}{\mu(B)}\frac{\mu(B)}{\mu(\wt{B})}\|f\|_{L^\infty(X)}=\eps\|f\|_{L^\infty(X)}.
\]
We define $g_2:=-\tilde{h}\in L^\infty(\{x\in \wt{B} : \tilde{f}(x)\neq 0\})$ and $h_2:=\tilde{g}$. 

Because $f$ vanishes everywhere outside $E$, it holds that 
\[
h_1\in  L^\infty(E).
\]
Using a similar reasoning, it holds that
\[
g_2\in L^\infty(\wt{E}).
\]

We prove that $h_2T^*(-g_2)=-h_2T^*g_2$. For this, note that because $h_2\in L^\infty(E)$, we have for every $x\in X\setminus E$ that
\[
h_2(x)T^*(-g_2)(x)=0=-h_2(x)T^*g_2(x).
\]
On the other hand, because $E\subset X\setminus \supp(g_2)=X\setminus \supp(-g_2)$, we have for every $y\in E$ that 
\[
T^*(-g_2)(y)=-\int_{X}K^*(y,x)g_2(x) d \mu(x)=-T^*(g_2)(y).
\]

Now we see that
\begin{align*}
f&=g_1Th_1-h_1T^*g_1+\tilde{f} \\
&=g_1Th_1-h_1T^*g_1+\tilde{g}T^*\tilde{h}-\tilde{h}T\tilde{g}+\tilde{\tilde{f}} \\
&=g_1Th_1-h_1T^*g_1+h_2T^*(-g_2)+g_2Th_2+\tilde{\tilde{f}} \\
&=g_1Th_1-h_1T^*g_1-h_2T^*g_2+g_2Th_2+\tilde{\tilde{f}} \\
&=\sum_{i=1}^2(g_iTh_i-h_iT^*g_i)+\tilde{\tilde{f}}.
\end{align*}
The proof is complete.
\end{proof}

We apply the factorisation to $\omega$-Calder\'{o}n-Zygmund kernels that satisfy \eqref{eq-1.1opp}, because they are admissible for the factorisation. As a result we prove our main tool (Proposition \ref{prop:oikarisawfformulationforcalderonzygmundxdmu}) to show a necessary condition for the boundedness of commutators.

\begin{proof}[Proof of Proposition \ref{prop:oikarisawfformulationforcalderonzygmundxdmu}]
For our purposes, let us just define $T^*\colon L^1_{\rm bs}(X) \to L^0(X)$ by setting
\[
T^*f(x):=\begin{cases}
    \int_{X} K(y,x) f(y) d \mu(y), &\text{if } x\notin \supp(f), \\
    0, &\text{if } x\in \supp(f).
\end{cases}
\]
Then clearly $T^*\in \mathrm{SIO}(K^*,L^1_{\rm bs}(X))$.

Let $A\geq 2A_0^2+A_0$ and let $\eps_A$ be as in Proposition \ref{prop:propertiesofkernelsxdmu}. By Proposition \ref{prop:propertiesofkernelsxdmu}, there exists a constant $\xi=\xi(c_0,\bar{C},c_K,C_\mu,A_0)$ and a ball $\wt{B}_A$ such that $(K,\xi,A,\eps_A,B,\wt{B}_A)$ is admissible in the sense of Lemma \ref{lemma:hytonenlemma2.6firstversionxdmu} and $d(B,\wt{B}_A)\approx_{\bar{C},A}r$. Let $u=u(c,\xi,C_\mu,A_0)$ be as in Lemma \ref{lemma:hytonenlemma2.6firstversionxdmu}. By Proposition \ref{prop:propertiesofkernelsxdmu} we may choose $A$ so large that 
\[
\eps_A\leq u.
\]
In particular, $A$ may be chosen so that it depends at most on parameters $c$, $c_0$, $\bar{C}$, $c_K$, $C_\mu$, and $A_0$. Let us fix such an $A$ for now and let us simply denote $\wt{B}=\wt{B}_A$. 

Lemma \ref{lemma:hytonenlemma2.6firstversionxdmu} then applies to $(K,\xi,A,\eps_A,B,\wt{B})$, $T$, $T^*$ and $c$. Suppose $E\subset B$ and $\wt{E}\subset \wt{B}$ are such that $\mu(B)\leq c\mu(E)$ and $\mu(\wt{B})\leq c\mu(\wt{E})$. We may write as follows using a complex function $\alpha\in L^0(X)$ such that $|b-\ave{b}_E|=2(b-\ave{b}_E)\alpha$ and $|\alpha|\equiv 1/2$:
\begin{align*}
\int_{E}|b-\ave{b}_E| d \mu = 2\int_E (b-\ave{b}_E) \alpha d \mu &= 2\int_E (b-\ave{b}_E)(\alpha-\ave{\alpha}_E) d \mu \\
&=2\int_E b\underbrace{(\alpha-\ave{\alpha}_E)\chi_E}_{=:f} d \mu \\
&=2\big|\int_E bf d \mu\big|
\end{align*}
where $f\in L_0^\infty(E)$ and $\|f\|_{L^\infty(X)}\leq 1$. We deduced (essentially because $L^\infty$ is the dual of $L^1$) that 
\begin{equation}\label{eq:dualformatofoscillation}
\frac{1}{2}\int_{E}|b-\ave{b}_E| d \mu = \big|\int_{E} bf d \mu\big|.
\end{equation}
We may apply the decomposition of Lemma \ref{lemma:hytonenlemma2.6firstversionxdmu} to write
\[
f=\sum_{i=1}^2(g_iTh_i-h_iT^*g_i)+\tilde{\tilde{f}},
\]
where $\tilde{\tilde{f}}\in L_0^\infty(E)$, $g_i\in L^\infty(\wt{E})$ and $h_i\in L^\infty(E)$ satisfy 
\[
g_1=\chi_{\wt{E}}, \quad h_2=\chi_E, 
\]
\[
\|h_1\|_{L^\infty(X)}\lesssim_{c,c_0,\bar{C},c_K,C_\mu,A_0} \|f\|_{L^\infty(X)}\leq 1, \quad \|g_2\|_{L^\infty(X)}\lesssim_{c,c_0,\bar{C},c_K,C_\mu,A_0} \|f\|_{L^\infty(X)} \leq 1,
\]
and
\[
\|\tilde{\tilde{f}}\|_{L^\infty(X)}\lesssim_{c,c_0,\bar{C},c_K,C_\mu,A_0}\eps_A\|f\|_{L^\infty(X)}\leq \eps_A.
\]
We write
\begin{align*}
\int_E bf d \mu = \int_X bf d \mu &= \sum_{i=1}^2\int_X bg_iTh_i d \mu -\sum_{i=1}^2\int_X bh_iT^*g_i d \mu + \int_X b\tilde{\tilde{f}} d \mu \\
&\overset{\text{Cor. }\ref{cor:tandtstarareadjointsxdmu}}{=} \sum_{i=1}^2\int_X bg_iTh_i d \mu -\sum_{i=1}^2\int_X g_i T(bh_i) d \mu + \int_X b\tilde{\tilde{f}} d \mu \\
&= \sum_{i=1}^2 \int_X g_i(bTh_i-T(bh_i)) d \mu +\int_X b\tilde{\tilde{f}} d \mu \\
&= \sum_{i=1}^2\int_{\wt{E}} g_i[b,T]h_i d \mu + \int_E b\tilde{\tilde{f}} d \mu.
\end{align*}
Plugging this information into \eqref{eq:dualformatofoscillation}, we get
\[
\frac{1}{2}\int_E |b-\ave{b}_E| d \mu \leq \sum_{i=1}^2 \big| \int_{\wt{E}}g_i[b,T]h_i d \mu \big| + \big| \int_E b\tilde{\tilde{f}} d \mu \big|.
\]
Note that since $\int_E \tilde{\tilde{f}} d \mu=0$, 
\begin{align*}
\big| \int_E b\tilde{\tilde{f}} d \mu \big| = \big| \int_E (b-\ave{b}_E)\tilde{\tilde{f}} d \mu \big| &\leq \|\tilde{\tilde{f}}\|_{L^\infty(X)} \int_E |b-\ave{b}_E| d \mu \\
&\leq D(c,c_0,\bar{C},c_K,C_\mu,A_0)\eps_A \int_E |b-\ave{b}_E| d \mu.
\end{align*}
Therefore we have
\begin{align*}
&\frac{1}{2}\int_E |b-\ave{b}_E| d \mu \\
&\leq \sum_{i=1}^2 \big| \int_{\wt{E}}g_i[b,T]h_i d \mu \big| + D(c,c_0,\bar{C},c_K,C_\mu,A_0)\eps_A \int_E |b-\ave{b}_E| d \mu.
\end{align*}
Motivated by this, we further require $A$ to be so big that \[
\eps_A\leq 4^{-1}D(c,c_0,\bar{C},c_K,C_\mu,A_0)^{-1}.
\]
The parameters that $A$ at most depends on remain unchanged. Then we can absorb the oscillation on the right-hand side of the previous inequality to the left-hand side and get 
\[
\frac{1}{4}\int_E |b-\ave{b}_E| d \mu \leq \sum_{i=1}^2 \big| \int_{\wt{E}}g_i[b,T]h_i d \mu \big|.
\]
If we then rename the functions $g_i$ and $h_i$ so that their subscripts denote which $L^\infty(?)$ space they are in, we get the claim.
\end{proof}

By assuming non-degeneracy \eqref{eq-1.1}, we get an analogous consequence of the approximate weak factorisation as follows. The proof is given through Proposition \ref{prop:oikarisawfformulationforcalderonzygmundxdmu} and a duality argument.

\begin{proposition}[Bounding oscillations with \eqref{eq-1.1}]\label{prop:awfformulationforinvnondegcalderonzygmund}
Suppose $K$ is a non-degenerate $\omega$-Calder\'{o}n-Zygmund kernel in $X$ (that is, in the sense of \eqref{eq-1.1}) and suppose $T\in \mathrm{SIO}(K,L^1_{\rm bs}(X))$, $b\in L^1_{\mathrm{loc}}(X)$ and $c\geq 1$. Let $B$ be a ball in $X$ with radius $r$.

Then there exists a ball $\wt{B}$ that has the same radius $r$ as $B$, lies at distance $d(B,\wt{B})\approx r$ from $B$ and the following holds: for any $E\subset B$ and $\wt{E}\subset \wt{B}$ such that $\mu(B)\leq c\mu(E)$ and $\mu(\wt{B})\leq c\mu(\wt{E})$, we have 
\begin{equation}
\int_{E}|b-\ave{b}_E| d \mu \lesssim \big|\int_{E}h_E[b,T]g_{\wt{E}} d \mu \big|+\big|\int_{E} g_E [b,T]h_{\wt{E}} d \mu \big|,
\end{equation}
where the auxiliary functions $g_E,h_E\in L^\infty(E)$ and $g_{\wt{E}},h_{\wt{E}}\in L^\infty(\wt{E})$ satisfy 
\[
g_E=\chi_E, \quad g_{\wt{E}}=\chi_{\wt{E}}, \quad \|h_E\|_{L^\infty(X)}\lesssim 1, \quad \|h_{\wt{E}}\|_{L^\infty(X)}\lesssim 1.
\]
The implied constants depend at most on the kernel parameters $c_0$, $\bar{C}$, $c_K$, the parameters $C_\mu$ and $A_0$ of the space $X$ and on $c$.
\end{proposition}
\begin{proof}
    By the non-degeneracy assumption of $K$ (translated to a statement on $K^*$), it holds that there exist positive constants $c_0$ and $\bar{C}$ such that for every $y \in X$ and $r>0$, there exists $x \in B(y, \bar{C} r) \backslash B(y, r)$ satisfying
    \[
    |K^*(x, y)| \geq \frac{1}{c_0 \mu(B(y, r))}.
    \]
    Thus the $\omega$-Calder\'{o}n-Zygmund kernel $K^*$ satisfies \eqref{eq-1.1opp} with the parameters $c_0$ and $C_0$.

    Let us define $T^*\colon L^1_{\rm bs}(X) \to L^0(X)$ by setting
    \[
    T^*f(x):=\begin{cases}
    \int_{X} K^*(x,y) f(y) d \mu(y), &\text{if } x\notin \supp(f), \\
    0, &\text{if } x\in \supp(f).
    \end{cases}
    \]
    Then clearly $T^*\in \mathrm{SIO}(K^*,L^1_{\rm bs}(X))$. Proposition \ref{prop:oikarisawfformulationforcalderonzygmundxdmu} applies to $K^*$, $T^*$, $b$, $c$ and $B$. Then there exists a ball $\wt{B}$ that has the same radius $r$ as $B$, lies at distance 
    \[
    d(B,\wt{B})\approx r
    \]
    from $B$ and the following holds: for any $E\subset B$ and $\wt{E}\subset \wt{B}$ such that $\mu(B)\leq c\mu(E)$ and $\mu(\wt{B})\leq c\mu(\wt{E})$, we have 
    \begin{align*}
    \int_{E}|b-\ave{b}_E| d \mu \lesssim \Big|\int_{\wt{E}}g_{\wt{E}}[b,T^*]h_E d \mu \Big| 
    +\Big|\int_{\wt{E}}h_{\wt{E}}[b,T^*]g_E d \mu \Big|,
    \end{align*}
    where the auxiliary functions $g_E,h_E\in L^\infty(E)$ and $g_{\wt{E}},h_{\wt{E}}\in L^\infty(\wt{E})$ satisfy 
    \[
    g_E=\chi_E, \quad g_{\wt{E}}=\chi_{\wt{E}}
    \] 
    \[
    \|h_E\|_{L^\infty(X)}\lesssim 1, \quad \|h_{\wt{E}}\|_{L^\infty(X)}\lesssim 1.
    \]
    Note that $T\in \mathrm{SIO}((K^*)^*,L^1_{\mathrm{bs}}(X))$, because $(K^*)^*=K$. Because $d(B,\wt{B})>0$, Corollary \ref{cor:tandtstarareadjointsxdmu} about the adjointness of $T$ and $T^*$ applies, and we get 
    \begin{align*}
    \int_{\wt{E}}g_{\wt{E}}[b,T^*]h_E d \mu &=\int_X T^*h_E \cdot b\cdot g_{\wt{E}} d \mu - \int_X g_{\wt{E}} \cdot T^*(bh_E) d \mu \\
    &= \int_X h_E \cdot T(bg_{\wt{E}}) d \mu - \int_X Tg_{\wt{E}} \cdot b\cdot h_E d \mu \\
    &=-\int_{E}h_E[b,T]g_{\wt{E}} d \mu
    \end{align*}
    and similarly
    \begin{align*}
    \int_{\wt{E}}h_{\wt{E}}[b,T^*]g_E d \mu=-\int_{E}g_E[b,T]h_{\wt{E}} d \mu.
    \end{align*}
    Substituting these to our upper bound on the oscillation, we get
    \begin{align*}
    \int_{E}|b-\ave{b}_E| d \mu \lesssim \Big|\int_{E}h_E[b,T]g_{\wt{E}} d \mu \Big| 
    +\Big|\int_{E}g_E[b,T]h_{\wt{E}} d \mu \Big|.
    \end{align*}
    \end{proof}

We then demonstrate how to get the lower bound for the commutator norm when one assumes non-deneracy \eqref{eq-1.1}. An analogous lower bound may be derived when condition \eqref{eq-1.1opp} is assumed. The proof is similar and we only note that the dependence on the weight constants in the upper bound of $\|b\|_{\BMO_\nu^\alpha(X)}$ will then be $[\lambda_1]_{A_{p,p}}[\lambda_2]_{A_{q,q}}^2$, as opposed to $[\lambda_1]_{A_{p,p}}^2[\lambda_2]_{A_{q,q}}$. We leave deriving that to the interested reader.

\begin{theorem}\label{thm:commutatorlowerboundawfversion}
Suppose $K$ is a non-degenerate $\omega$-Calder\'{o}n-Zygmund kernel in $X$ (that is, in the sense of \eqref{eq-1.1}), $T\in \mathrm{SIO}(K,L^1_{\rm bs}(X))$ and $b\in L^1_{\mathrm{loc}}(X)$. Suppose $1<p,q<\infty$, $\lambda_1\in A_{p,p}$ and $\lambda_2\in A_{q,q}$. Then if there exists a constant $\Theta\in(0,\infty)$ so that for all $f\in L^\infty_{\rm bs}(X)$ we have
\[
\|[b,T]f\|_{L^q_{\lambda_2}(X)}\leq \Theta\|f\|_{L^p_{\lambda_1}(X)},
\]
then $b\in \BMO_\nu^\alpha(X)$ and
\[
\|b\|_{\BMO_\nu^\alpha(X)}\lesssim \Theta[\lambda_1]_{A_{p,p}}^2[\lambda_2]_{A_{q,q}},
\]
where $\alpha/\QQ=1/p-1/q$ and $\nu=(\lambda_1 / \lambda_2)^\frac{1}{1/p+1/q'}$.

The implied constant depends at most on parameters $c_0$, $\bar{C}$, $c_K$, $C_\mu$ and $A_0$.
\end{theorem}
\begin{proof}
Let $B=B(y_0,r)$ be a ball in $X$. Applying Proposition \ref{prop:awfformulationforinvnondegcalderonzygmund} with $c=1$ and $B$, we get a ball $\wt{B}=B(x_0,r)$ that lies at distance $d(B,\widetilde{B})\approx r$ from $B$ and satisfies the property stated in that proposition. In particular, then there exist functions $h_B$ and $h_{\wt{B}}$ as in that proposition such that
\begin{align*}
\int_B|b-\ave{b}_B| d \mu &\lesssim \big|\langle [b,T]\chi_{\wt{B}},h_B \rangle \big|+\big|\langle [b,T]h_{\wt{B}},\chi_B \rangle \big| \\
&= \big|\langle \lambda_2[b,T]\chi_{\wt{B}},h_B\lambda_2^{-1} \rangle \big|+\big|\langle \lambda_2[b,T]h_{\wt{B}},\chi_B\lambda_2^{-1} \rangle \big| \\
&\leq \|[b,T]\chi_{\wt{B}}\|_{L^q_{\lambda_2}(X)}\|h_B\|_{L^{q'}_{\lambda_2^{-1}}(X)} + \|[b,T]h_{\wt{B}}\|_{L^q_{\lambda_2}(X)} \|\chi_B\|_{L^{q'}_{\lambda_2^{-1}}(X)} \\
&\leq \Theta\bigg(\|\chi_{\wt{B}}\|_{L^p_{\lambda_1}(X)}\|h_B\|_{L^{q'}_{\lambda_2^{-1}}(X)}+\|h_{\wt{B}}\|_{L^p_{\lambda_1}(X)}\|\chi_B\|_{L^{q'}_{\lambda_2^{-1}}(X)}\bigg) \\
&\lesssim \Theta\bigg(\|\chi_{\wt{B}}\|_{L^p_{\lambda_1}(X)}\|\chi_B\|_{L^{q'}_{\lambda_2^{-1}}(X)}+\|\chi_{\wt{B}}\|_{L^p_{\lambda_1}(X)}\|\chi_B\|_{L^{q'}_{\lambda_2^{-1}}(X)}\bigg) \\
&\lesssim \Theta \lambda_1^p(\wt{B})^\frac{1}{p}\lambda_2^{-q'}(B)^\frac{1}{q'}.
\end{align*}
It remains to show that $\lambda_1^p(\wt{B})^\frac{1}{p}\lambda_2^{-q'}(B)^\frac{1}{q'}$ can be bounded by $\nu(B)^{1+\alpha/\QQ}$ with an implicit constant depending only on the parameters it is allowed to depend on.

Because $\lambda_1\in A_{p,p}$, we have $\lambda_1^p\in A_p$. Note that for some constant $c' > 1$, it holds that
\[
\wt{B}\subset B(y_0,c'r)
\]
and
\[
B\subset B(y_0,c'r).
\]
Then by the doubling property of $\lambda_1^p$ (see the remark after Definition \ref{defi:apweightsmuckenhoupt}), we get
\[
\lambda_1^p(\wt{B})^\frac{1}{p} \leq \lambda_1^p(B(y_0,c'r))^\frac{1}{p} \lesssim [\lambda_1^p]_{A_p}^\frac{1}{p} \lambda_1^p(B)^\frac{1}{p} = [\lambda_1]_{A_{p,p}} \lambda_1^p(B)^\frac{1}{p}.
\]
Inputting this estimate to our first inequality chain, we get
\[
\int_B|b-\ave{b}_B| d \mu \lesssim \Theta [\lambda_1]_{A_{p,p}} \lambda_1^p(B)^\frac{1}{p}\lambda_2^{-q'}(B)^\frac{1}{q'}.
\]
Thus the claim follows from Lemma \ref{lemmaS2}.
\end{proof}

From subsection \ref{commutsdefinition}, it follows that under the assumptions of Theorem \ref{thm2}, it holds that 
\[
T\in \mathrm{SIO}(K,L^1_{\rm bs}(X)).
\]
Indeed, Theorem \ref{thm2} is thus a corollary to Theorem \ref{thm:commutatorlowerboundawfversion}.

\section{Discussion}\label{discussion}
We come back to the topic briefly discussed in the introduction. Namely, we present alternative assumptions for the $\mu$-measurable sets in the definition of a space of homogeneous type $(X, d, \mu)$.  So, the substance of this discussion is to inform the interested reader of certain alternative assumptions. We will leave the following question open:
\begin{align*}
\textbf{Question: } &\text{Do our main results Theorems \ref{thm1} and \ref{thm2} remain valid } \\
&\text{under any of the two weaker assumptions presented shortly?}
\end{align*}
In what follows, we explain why it is likely that the answer to the above question is positive.

As is to be expected by now, we suppose for this discussion that $X$ is a set $d$ is a quasi-metric in $X$, $\mathfrak{M}$ is a $\sigma$-algebra in $X$ that contains all balls and that $\mu \colon \mathfrak{M} \to  [0,\infty]$ is a doubling measure that satisfies $\mu(X)=\infty$ and $\mu(\{x_0\})=0$ for every $x_0\in X$. As explained, we are interested in the conditions one can set on $\mathfrak{M}$.

Let $\mathrm{Bor}_d(X)$ be the Borel $\sigma$-algebra in $X$, that is, the smallest $\sigma$-algebra in $X$ that contains all open sets in $X$. It is known (based on the $\sigma$-finiteness of $X$ and the geometric doubling property (Lemma \ref{lemma:geometricdoubling}) which in turn follows from the assumed doubling property of $\mu$) that each open set is a countable union of balls. Thus our minimal assumption that balls are in $\mathfrak{M}$ in fact already implies that 
\begin{equation}\label{eq:readilyborelaremeas}
    \mathrm{Bor}_d(X) \subset \mathfrak{M}.
\end{equation} 

Recall that $E\bigtriangleup B$ means the symmetric difference of sets $E$ and $B$. To summarise the following discussion, we have the following three natural condition candidates for our $\mu$-measurable sets in the definition of a space of homogenous type:
\begin{itemize}
    \item[(1)] $\mathfrak{M} = \mathrm{Bor}_d(X)$;
    \item[(2)] For every $E\in \mathfrak{M}$, there exists $B\in \mathrm{Bor}_d(X)$ with the property that $E\subset B$ and $\mu(E) = \mu(B)$;
    \item[(3)] For every $E\in \mathfrak{M}$, there exists $B\in \mathrm{Bor}_d(X)$ with the property that $\mu(E\bigtriangleup B) = 0$ (see \cite{AM2015} by Alvarado and Mitrea).
\end{itemize}
Note that (2) and (3) are well-defined because of \eqref{eq:readilyborelaremeas}. Condition (1) is the one we adopt in this paper, as seen in the introduction. It is easy to see that $(1) \Rightarrow (2) \Rightarrow (3)$. Thus (3) is the least restrictive of the three conditions. We do not show whether or not definitions of a space of homogeneous type based on (2) and (3) are different. Also, we are not aware of a related reference to point to. The typical Euclidean setting on Borel sets versus on Lebesgue sets is the obvious example that shows definitions based on (1) and (2) are different.

Condition (1) is what many authors might enunciate as ``$\mu$ is a Borel measure''. On the other hand, many authors would use this exact phrase to mean instead that $\mathrm{Bor}_d(X) \subset \mathfrak{M}$. It is probably safe to proclaim that this double meaning has lead to many misunderstandings in communication in general throughout history. Be as it may, condition (1) is often used in conjunction with spaces of homogeneous type. The reason to assume condition (1) might be to avoid technicalities and thus leave the modifications needed for a more general assumption to the interested reader. Condition (1) leaves out natural candidates such as the Euclidean setting $(\R^n, \|\cdot\|_{\R^n}, \mathrm{Leb}(\R^n))$, where $\mathrm{Leb}(\R^n)$ is the collection of Lebesgue-measurable sets.

Condition (2) is often referred to as ``$\mu$ is Borel-regular''. This is also an often used assumption in conjunction with spaces of homogeneous type.

Condition (3) is taken from \cite{AM2015}. It is there referred to as ``$\mu$ is Borel-semiregular''. Under this condition on $\mu$, a sharp version of Lebesgue's differentiation theorem (\cite[Theorem 3.14]{AM2015}), in a sense, is proved whenever the quasi-metric possesses certain regularity. A conclusion then is that the Borel-semiregularity of $\mu$ is necessary for the condition that for every $f\in L^1_\mathrm{loc}(X)$,
\begin{equation}\label{eq:lebdifthmexample}
\lim_{r\to 0^+} \frac{1}{\mu(B(x,r))}\int_{B(x,r)} f(y) d \mu(y) = f(x)
\end{equation}
is true for $\mu$-almost every $x\in X$; for quasi-metrics with certain regularity. In fact, it is contained in \cite[Theorem 3.14]{AM2015} that Borel-semiregularity is also sufficient for \eqref{eq:lebdifthmexample} (and its typical stronger form) to hold; again, for quasi-metrics with certain regularity. For the skipped details, we refer to \cite{AM2015}. Condition (3) is perhaps not as widely known in the context of spaces of homogeneous type as conditions (1) and (2) are. Hence particular background theory under condition (3) might be more difficult to find than that based on condition (2), for example.

Lebesgue's differentiation theorem plays an important role in singular integral theory. Namely, when making the Calder\'{o}n-Zygmund decomposition $f=g+b$, a decomposition of the upper level set of the maximal function is made. Lebesgue's differentiation theorem guarantees then that $g$ is bounded outside the upper level set by averages taken over neighbourhoods of the outside points. Lebesgue's differentiation theorem also plays a role in sparse estimates of singular integrals and their commutators (see e.g. \cite[Lemma 3.1]{LOR2017}). By experience, Lebesgue's differentiation theorem seems to be the part of the background theory of our main results that requires the most regularity from measurable sets. This is why it is likely that the answer to the above question is positive, at least if $d$ possesses certain regularity.

What if the quasi-metric $d$ does not possess certain regularity? It is likely that this case can be handled by reducing to regular quasi-metrics by equivalence. Namely, $d$ is equivalent to a regular quasi-metric that satisfies the stronger form of \eqref{eq:lebdifthmexample} and all balls in that quasi-metric are in $\mathfrak{M}$. So, by switching to a larger ball in the regular quasi-metric with comparable average, the stronger form of \eqref{eq:lebdifthmexample} for the regular quasi-metric implies the stronger form of \eqref{eq:lebdifthmexample} for $d$. We conclude that Borel-semiregularity (a property of $X$ which is invariant under equivalent quasi-metrics) allows us to use \eqref{eq:lebdifthmexample} in any space of homogeneous type $(X,d,\mu)$ defined through condition (3). However, because of reasons related to the measurability of the Hardy-Littlewood maximal function, it might be more comfortable to use the equivalence of $d$ to a regular quasi-metric at some other point in the argument rather than at that point where Lebesgue's differentiation theorem is used. For our main results, the equivalence can probably be used at the very end (cf. the metamathematical principle in \cite{Ste2015}). By this, we mean proving the whole theorem for regular quasi-metrics first. This is why it is likely that the answer to the above question is positive, even if $d$ does not possess certain regularity.

\section*{Acknowledgements}
{J.S. thanks Tuomas Hyt\"onen for helpful and inspiring discussions on the approximate weak factorisation.

J.L. is supported by ARC DP 220100285. J.S. was supported by the Research Council of Finland through project 346314 (to Tuomas Hyt\"{o}nen; this project is part of the Finnish Centre of Excellence in Randomness and Structures) and projects 336323, 358180 (to Timo H\"{a}nninen). Z. G. is supported by Research Fund of Southwest University of Science and Technology 23zx7167.
}

\bibliographystyle{plain}
\bibliography{manuscriptreferences}

\begin{thebibliography}{10}

\bibitem{AM2015}
Ryan Alvarado and Marius Mitrea.
\newblock Hardy spaces on {A}hlfors-regular quasi metric spaces.
\newblock {\em Lecture Notes in Mathematics}, 2142, 2015.

\bibitem{B1985}
Steven Bloom.
\newblock A commutator theorem and weighted {BMO}.
\newblock {\em Transactions of the American Mathematical Society},
  292(1):103--122, 1985.

\bibitem{CHDD2020}
Yanping Chen, Qingquan Deng, and Yong Ding.
\newblock Commutators with fractional differentiation for second-order elliptic
  operators on $\mathbb{R}^n$.
\newblock {\em Communications in Contemporary Mathematics}, 22(02):1950010,
  2020.

\bibitem{CYZ2022}
Yanping Chen, Dunyan Yan, and Kai Zhu.
\newblock Necessary and {S}ufficient {C}onditions for the {B}ounds of the
  {C}ommutators with {F}ractional {D}ifferentiations and {BMO-Sobolev} {S}paces
  on {W}eighted {L}ebesgue space.
\newblock {\em The Journal of Geometric Analysis}, 32(3):103, 2022.

\bibitem{CRW1976}
Ronald~R Coifman, Richard Rochberg, and Guido Weiss.
\newblock Factorization theorems for {H}ardy spaces in several variables.
\newblock {\em Annals of Mathematics}, 103(3):611--635, 1976.

\bibitem{CW1971}
Ronald~R. Coifman and Guido Weiss.
\newblock {\em Analyse Harmonique Non-Commutative sur Certains Espaces
  Homogenes: Etude de Certaines Int{\'e}grales Singuli{\`e}res}.
\newblock Springer Berlin Heidelberg, Berlin, Heidelberg, 1971.

\bibitem{CW1977}
Ronald~R Coifman and Guido Weiss.
\newblock Extensions of {H}ardy spaces and their use in analysis.
\newblock {\em Bulletin of the American Mathematical Society}, 83(4):569--645,
  1977.

\bibitem{DGKLWY2021}
Xuan~Thinh Duong, Ruming Gong, Marie-Jose~S Kuffner, Ji~Li, Brett~D Wick, and
  Dongyong Yang.
\newblock Two weight commutators on spaces of homogeneous type and
  applications.
\newblock {\em The Journal of Geometric Analysis}, 31(1):980--1038, 2021.

\bibitem{DY2003}
Xuan~Thinh Duong and Lixin Yan.
\newblock Commutators of {BMO} functions and singular integral operators with
  non-smooth kernels.
\newblock {\em Bulletin of the Australian Mathematical Society},
  67(2):187--200, 2003.

\bibitem{HLS2023}
Timo~S. H\"{a}nninen, Emiel Lorist, and Jaakko Sinko.
\newblock Weighted ${L}^p\to {L}^q$-boundedness of commutators and paraproducts
  in the {B}loom setting, 2023.

\bibitem{HLW2016}
Irina Holmes, Michael~T Lacey, and Brett~D Wick.
\newblock Bloom's inequality: commutators in a two-weight setting.
\newblock {\em Archiv der Mathematik}, 106:53--63, 2016.

\bibitem{HLW2017}
Irina Holmes, Michael~T Lacey, and Brett~D Wick.
\newblock Commutators in the two-weight setting.
\newblock {\em Mathematische Annalen}, 367:51--80, 2017.

\bibitem{HOS2023}
Tuomas Hyt{\"o}nen, Tuomas Oikari, and Jaakko Sinko.
\newblock Fractional {B}loom boundedness and compactness of commutators.
\newblock {\em Forum Mathematicum}, 35(3):809--830, 2023.

\bibitem{H2012}
Tuomas~P Hyt{\"o}nen.
\newblock The sharp weighted bound for general {C}alder{\'o}n--{Z}ygmund
  operators.
\newblock {\em Annals of mathematics}, pages 1473--1506, 2012.

\bibitem{Hyt2021}
Tuomas~P Hyt{\"o}nen.
\newblock The {{\(L^p\)}}-to-{{\(L^q\)}} boundedness of commutators with
  applications to the {Jacobian} operator.
\newblock {\em Journal de Math{\'e}matiques Pures et Appliqu{\'e}es},
  156:351--391, 2021.

\bibitem{J1978}
Svante Janson.
\newblock Mean oscillation and commutators of singular integral operators.
\newblock {\em Arkiv f{\"o}r Matematik}, 16(1):263--270, 1978.

\bibitem{KL2001}
Steven~G Krantz and Song-Ying Li.
\newblock Boundedness and {C}ompactness of {I}ntegral {O}perators on {S}paces
  of {H}omogeneous {T}ype and {A}pplications, {I}.
\newblock {\em Journal of mathematical analysis and applications},
  258(2):629--641, 2001.

\bibitem{LL2022}
Michael Lacey and Ji~Li.
\newblock Compactness of commutator of {R}iesz transforms in the two weight
  setting.
\newblock {\em J. Math. Anal. Appl.}, 508(1):Paper No. 125869, 11, 2022.

\bibitem{LLO2024}
Andrei~K Lerner, Emiel Lorist, and Sheldy Ombrosi.
\newblock Bloom weighted bounds for sparse forms associated to commutators.
\newblock {\em Mathematische Zeitschrift}, 306(4):73, 2024.

\bibitem{LOR2017}
Andrei~K Lerner, Sheldy Ombrosi, and Israel~P Rivera-R{\'\i}os.
\newblock On pointwise and weighted estimates for commutators of
  {C}alder{\'o}n--{Z}ygmund operators.
\newblock {\em Advances in Mathematics}, 319:153--181, 2017.

\bibitem{LOR2019}
Andrei~K Lerner, Sheldy Ombrosi, and Israel~P Rivera-R{\'\i}os.
\newblock Commutators of singular integrals revisited.
\newblock {\em Bulletin of the London Mathematical Society}, 51(1):107--119,
  2019.

\bibitem{L2022}
Kangwei Li.
\newblock Multilinear commutators in the two-weight setting.
\newblock {\em Bulletin of the London Mathematical Society}, 54(2):568--589,
  2022.

\bibitem{MS1979}
Roberto~A Mac{\'\i}as and Carlos Segovia.
\newblock Lipschitz functions on spaces of homogeneous type.
\newblock {\em Advances in Mathematics}, 33(3):257--270, 1979.

\bibitem{Ne1957}
Zeev Nehari.
\newblock On bounded bilinear forms.
\newblock {\em Annals of Mathematics}, 65(1):153--162, 1957.

\bibitem{P2000}
Stefanie Petermichl.
\newblock Dyadic shifts and a logarithmic estimate for {H}ankel operators with
  matrix symbol.
\newblock {\em Comptes Rendus de l'Acad{\'e}mie des Sciences-Series
  I-Mathematics}, 330(6):455--460, 2000.

\bibitem{PS2007}
Gladis Pradolini and Oscar Salinas.
\newblock Commutators of singular integrals on spaces of homogeneous type.
\newblock {\em Czechoslovak Mathematical Journal}, 57:75--93, 2007.

\bibitem{Ste1993}
Elias~M Stein and Timothy~S Murphy.
\newblock {\em Harmonic analysis: real-variable methods, orthogonality, and
  oscillatory integrals}, volume~3.
\newblock Princeton University Press, 1993.

\bibitem{Ste2015}
Krzysztof Stempak.
\newblock On some structural properties of spaces of homogeneous type.
\newblock {\em Taiwanese Journal of Mathematics}, 19(2):603--613, 2015.

\end{thebibliography}
\end{document}